%% file: arXivWaveRORrev1.tex
\documentclass[11pt, a4paper]{amsart}
\input{MKprencls}

\renewcommand{\dom}[1]{{\mathcal D (#1)}}
\renewcommand{\ker}{\mathcal{N}}

\begin{document}

\title[Approximate robust control of boundary control systems]{Approximate robust output regulation of boundary control systems}

\author{Jukka-Pekka Humaloja}
\address{Tampere University of Technology, Mathematics, P.O. Box 553, 33101, Tampere, Finland}
\email{jukka-pekka.humaloja@tut.fi}
\email{lassi.paunonen@tut.fi}
\author{Mikael Kurula}
\address{\AA bo Akademi University, Mathematics and Statistics, Domkyrkotorget 1, 20500 \AA bo, Finland}
\email{mkurula@abo.fi}
\author{Lassi Paunonen}

\thispagestyle{plain}

\begin{abstract}
We extend the internal model principle for systems with boundary control and boundary observation, and construct a robust controller for this class of systems. However, as a consequence of the internal model principle, any robust controller for a plant with infinite-dimensional output space necessarily has infinite-dimensional state space. We proceed to formulate the approximate robust output regulation problem and present a finite-dimensional controller structure to solve it. Our main motivating example is a wave equation on a bounded multidimensional spatial domain with force control and velocity observation at the boundary. In order to illustrate the theoretical results, we construct an approximate robust controller for the wave equation on an annular domain and demonstrate its performance with numerical simulations.
\end{abstract}



\thanks{The research is supported by the Academy of Finland Grant number 310489 held by L. Paunonen. L. Paunonen is funded by the Academy of Finland grant number 298182.}

\maketitle
\section{Introduction}

Intuitively speaking, the problem of output regulation of a given plant amounts to designing an output feedback controller which stabilizes the plant, and in addition the output of the plant tracks a given reference signal in spite of a given disturbance signal. If a single controller solves the output regulation problem for the plant and also for small perturbations of the plant, and for more or less arbitrary reference and disturbance signals, then the controller is said to solve the \emph{robust} output regulation problem. See the beginning of \S\ref{sec:ROR} for exact definitions of these concepts.

Output tracking and disturbance rejection have been studied actively in the literature for distributed parameter systems with bounded control and observation operators \cite{ByrLau00,ImmPoh06b,NatGil14,Deu15,XuDub16} and robust controllers have been constructed for classes of systems with unbounded control and observation operators, such as well-posed \cite{StafBook} and regular \cite{WeissFeedback94} systems, in \cite{HamPoh00,RebWei03,BouHad09,Pau16a}. The key in designing robust controllers is the \emph{internal model principle} which in its classical form states that a controller can solve the robust output regulation problem only if it contains $p$ copies of the dynamics of the exosystem, where $p$ is the dimension of the output space of the plant. The internal model principle was first presented for finite-dimensional linear plants by Francis and Wonham \cite{FraWon75a} and Davison \cite{Dav76}. The principle was later generalized to infinite-dimensional linear systems in \cite{Pau16a, PauPoh10,PauPoh14} under the assumption that the plant is regular.

In this paper, we focus on output regulation for boundary controlled systems with boundary observation. Our motivating example is a wave equation on a multidimensional spatial domain, with force control and velocity observation on a part of the boundary. This $n$-D wave system is challenging from the robust control point of view since it is neither regular nor well-posed. Moreover, the output space of the wave system is infinite-dimensional and then the \emph{internal model principle} implies that any robust controller must also be infinite-dimensional. However, as the main contribution of this paper, we demonstrate that it is possible to achieve \textit{approximate} tracking of the reference signal in the sense that the difference between the output and the reference signal becomes small as $t\to \infty$. More precisely, we introduce a new finite-dimensional controller that solves the robust output regulation problem in this approximate sense, hence extending the recent results of \cite{Pau16barxiv} to continuous time. At the same time, we extend the class of reference signals that can be tracked. As a part of the construction of this controller, we present an upper bound for the regulation error.

The second main result of this paper is a generalization of the internal model principle presented in \cite{PauPoh10,PauPoh14} to boundary control systems that are not necessarily regular linear systems. The sufficiency of the internal model for achieving robust control has been presented in \cite{HumPau17}, albeit here our formulation is more general in terms of boundary controls and disturbances. The necessity of the internal model is a new result for boundary control systems. 

As our third main contribution we characterize and construct a minimal finite dimensional controller to solve the output regulation problem. Due to the reduced size of the controller, it does not have any guaranteed robustness properties. The controller concept was presented for regular linear systems in \cite{Pau16a}, and here we will generalize such controllers for boundary control systems.

In \S\ref{sec:firstord}, we present the wave equation and show how it fits into the abstract framework of the later sections. In \S\ref{sec:concontr}, we present the abstract plant, the exosystem and the controller (which is to be constructed), and reformulate the interconnection of these three systems as a regular input/state/output system. In \S\ref{sec:ROR}, we present the output regulation, the robust output regulation and the approximate robust output regulation problems, and present controller structures to solve them. A regulating controller without the robustness requirement is presented in \S\ref{ROR:nrrc}, and an approximate robust regulating controller is presented in \S\ref{ROR:arrc}. In \S\ref{ROR:ROR}, we present the internal model principle for boundary control systems, following which we present a precise robust regulating controller in \S\ref{ROR:rrc}. In \S\ref{sec:waveex}, we construct an approximate robust regulating controller for the wave equation on an annular domain and demonstrate its performance with numerical simulation. The paper is concluded in \S\ref{sec:concl}.

Here $\Lscr(X,Y)$ denotes the set of bounded linear operators from the normed space $X$ to the normed space $Y$. The domain, range, kernel, spectrum and resolvent of a linear operator $A$ are denoted by $\Dscr(A), \Rscr(A), \Nscr(A), \sigma(A)$ and $\rho(A)$, respectively. The right pseudoinverse of a surjective operator $P$ is denoted by $P^{[-1]}$.

\section{The wave equation}\label{sec:firstord} 

In this section, we describe the example which motivates the robust output regulation theory in this paper, a wave equation (the plant) on a bounded domain $\Omega\subset\R^n$ with force control and velocity observation at a part of the boundary. We try to keep the exposition brief; more details can be found in \cite{KuZw15,TuWeBook,KuZw12}.

Let $\Omega\subset\R^n$ be a bounded domain (an open connected set) with a Lipschitz-continuous boundary $\partial\Omega$ split into two parts $\Gamma_0,\Gamma_1$ such that $\overline{\Gamma_0}\cup\overline{\Gamma_1}=\partial\Omega$, $\Gamma_0\cap\Gamma_1=\emptyset$, and $\partial\Gamma_0,\partial\Gamma_1$ both have surface measure zero. We consider the wave equation 
\begin{equation}\label{eq:waveintro}
  \left\{
    \begin{aligned}
        \rho(\zeta)\frac{\partial^2 w}{\partial t^2} (\zeta,t) &= \Div \big(T(\zeta)\,\Grad w\,(\zeta,t)\big), \quad \zeta\in\Omega,\\
   u (\zeta,t) &= \nu\cdot T(\zeta)\,\Grad w(\zeta,t), \quad \zeta\in\Gamma_1,\\
    y (\zeta,t) &= \frac{\partial w}{\partial t}(\zeta,t), \qquad \zeta\in\Gamma_1,\\
    0 &= \frac{\partial w}{\partial t}(\zeta,t),\qquad \zeta\in\Gamma_0,~t>0 \\
    w(\cdot,0) &= w_0,\qquad \frac{\partial w}{\partial t}(\cdot,0)=w_1,
    \end{aligned}\right.
\end{equation}
where $w(\zeta,t)$ is the displacement from the equilibrium at the point $\zeta\in\Omega$ and time $t\geq0$, $\rho(\cdot)$ is the mass density, $T^*(\cdot)=T(\cdot)\in L^2(\Omega;\R^n)$ is Young's modulus and $\nu\in L^\infty(\partial\Omega;\R^n)$ is the unit outward normal at $\partial\Omega$. We require $\rho(\cdot)$ and $T(\cdot)$ to be essentially bounded from both above and below away from zero. Please note that the input $u$ is the force perpendicular to $\Gamma_1$ and the output $y$ is the velocity at $\Gamma_1$ while waves are reflected at the part $\Gamma_0$ of the boundary where the displacement is constant. 

In order to solve the robust output regulation problem for the wave system, we shall need to stabilize \eqref{eq:waveintro} exponentially using a viscous damper on $\Gamma_1$, 
which corresponds to the output feedback
$$
	u(\zeta,t)=-b^2(\zeta)\, y(\zeta,t),\qquad \zeta\in\Gamma_1,~t\geq0.
$$
This requires that we make some additional assumptions solely for the purpose of obtaining exponential stability (see \S\ref{sec:damper} below for more details). Additionally, to prove later on that the velocity observation on $\Gamma_1$ is admissible, we assume that
\begin{equation}\label{eq:dampass}
	\delta:=\inf_{\zeta\in\Gamma_1} b(\zeta)^2>0.
\end{equation}

\subsection{The wave equation as a formal boundary control system}

Our first step is to show that the wave equation on a bounded domain in $\R^n$ can be written as a boundary control system (BCS) in the sense of~\cite{CuZwbook}. To this end, we first write the wave equation
$$
	\rho(\zeta)\,\displaystyle \frac{\partial^2 w}{\partial t^2} (\zeta,t) = 
	\mathrm{div}\,\big(T(\zeta)\,\Grad w(\zeta,t)\big)
	\qquad\text{on}\quad \Omega\times\rplus
$$ 
in the first-order form (as an equality in $L^2(\Omega)^{n+1}$) 
\begin{equation}\label{eq:1orderWave}
  \displaystyle \ddt\bbm{\rho(\cdot)\,\dot w(\cdot,t)\\\Grad w(\cdot,t)}=\bbm{0&\mathrm{div}\\\Grad&0}
  \bbm{1/\rho(\cdot)&0\\0&T(\cdot)}\bbm{\rho(\cdot)\,\dot w(\cdot,t)\\\Grad w(\cdot,t)},
\end{equation}
where $\mathrm{div}$ denotes the (distribution) divergence operator and $\Grad$ is the (distribution) gradient. Hence, the state at any time is the pair of momentum and strain densities on $\Omega$. 

Under the standing assumptions on $\rho$ and $T$, the operator of multiplication by $\Hscr:=\sbm{1/\rho(\cdot)&0\\0&T(\cdot)}$ defines an inner product on $L^2(\Omega)^{n+1}$ via
$$
	\Ipdp{x}{z}_\Hscr:=\Ipdp{\Hscr x}{z}_{L^2(\Omega)^{n+1}}
$$
and $\Ipdp{\cdot}{\cdot}_{\mathcal{H}}$ is equivalent to $\Ipdp{\cdot}{\cdot}_{L^2(\Omega)^{n+1}}$.
The space $L^2(\Omega)^{n+1}$ equipped with this equivalent inner product is denoted by $X_\Hscr$ and will be used as the state space of the plant.

We next introduce some function spaces for the wave equation. The notation $H^1(\Omega)$ stands for the Sobolev space of all elements of $L^2(\Omega)$ whose distribution gradient lies in $L^2(\Omega)^n$ and $H^1(\Omega)$ is equipped with the graph norm of the gradient. Similarly $H^{\mathrm{div}}(\Omega)$ is the space of all elements of $L^2(\Omega)^n$ whose distribution divergence lies in $L^2(\Omega)$, equipped with the graph norm of $\mathrm{div}$. In order for \eqref{eq:1orderWave} to make sense as an equation in $L^2(\Omega)^{n+1}$, we need for every fixed $t\geq0$ that $\dot w(\cdot,t)\in H^1(\Omega)$, $\Grad w(\cdot,t)\in L^2(\Omega)$, and $T(\cdot)\,\Grad w(\cdot,t)\in H^{\mathrm{div}}(\Omega)$, or equivalently 
$$
	\bbm{\rho\, \dot w(t)\\\Grad w(t)}\in\Hscr^{-1}
		\bbm{H^1(\Omega)\\H^{\mathrm{div}}(\Omega)},\qquad t\geq0.
$$

If $\Gamma_0=\emptyset$, then the output $y$ lives in the fractional-order space $H^{1/2}(\partial\Omega)$ on the boundary of $\Omega$ (see, e.g.,\ \cite[\S 13.5]{TuWeBook} or \cite{KuZw12}). This space is important to us also when $\Gamma_0\neq\emptyset$, because the \emph{Dirichlet trace} $\gamma_0$ maps $H^1(\Omega)$ continuously onto $H^{1/2}(\partial\Omega)$. Indeed, we set
$$
\begin{aligned}
	\Wscr &:= \set {w\in H^{1/2}(\partial\Omega)\Bigmid w\big|_{\Gamma_0}=0}
	\qquad\text{with}\\
	\|w\|_\Wscr&:=\left\|\gamma_0^{[-1]}\,w\right\|_{H^1(\Omega)},
\end{aligned}
$$
where $|$ denotes the restriction to a given subdomain in the appropriate sense and 
$$
	\gamma_0^{[-1]}:=\gamma_0\big|_{\Ker{\gamma_0}^\perp}^{-1}
		\in\Lscr\big(H^{1/2}(\partial\Omega);H^1(\Omega)\big).
$$
Moreover, we introduce
$$
	H^1_{\Gamma_0}(\Omega) := \set{g\in H^1(\Omega)\bigmid g\big|_{\Gamma_0}=0},
$$
with the norm inherited from $H^1(\Omega)$. This setup makes both $\Wscr$ and $H^1_{\Gamma_0}(\Omega)$ Hilbert spaces; indeed, $H^{1/2}(\partial\Omega)$ is continuously embedded into $L^2(\partial\Omega)$ by \cite[(13.5.3)]{TuWeBook}, and so $H^1_{\Gamma_0}(\Omega)$ is the kernel of $P_{\Gamma_0}\gamma_0\in\Lscr\big(H^1(\Omega),L^2(\partial\Omega)\big)$, where $P_{\Gamma_0}$ is the orthogonal projection onto $L^2(\Gamma_0)$ in $L^2(\partial\Omega)$. This proves that $H^1_{\Gamma_0}(\Omega)$ is a Hilbert space, and moreover, $\gamma_0$ maps the Hilbert space $H^1_{\Gamma_0}(\Omega)\ominus \Ker{\gamma_0}$ unitarily onto $\Wscr$ which is then also complete.

The embedding $\iota:\Wscr\to L^2(\Gamma_1)$ is continuous, because $\iota=P_{\Gamma_1}
\widetilde\iota\gamma_0\gamma_0^{[-1]}$, where $\widetilde\iota$ is the continuous embedding of $H^{1/2}(\partial\Omega)$ into $L^2(\partial\Omega)$. The embedding is also dense by \cite[Thm 13.6.10]{TuWeBook}, so that we may define $\Wscr'$ as the dual of $\Wscr$ with pivot space $L^2(\Gamma_1)$ (see \cite[\S2.9]{TuWeBook}). Then in particular
$$
	\Ipdp\omega w_{\Wscr',\Wscr}=\Ipdp\omega w_{L^2(\Gamma_1)},\qquad
		\omega\in L^2(\Gamma_1),\, w\in\Wscr.
$$
Thm 1.8 in Appendix 1 of \cite{KuZw15} states that the \emph{restricted normal trace} $\gamma_\perp h:=(\nu\cdot\gamma_0 h)\big|_{\Gamma_1}$, $h\in H^1(\Omega)^n$, has a unique extension to a continuous operator (still denoted by $\gamma_\perp$) that maps $H^{\mathrm{div}}(\Omega)$ \emph{onto} $\Wscr'$. Please note that $\gamma_\perp$ is \emph{not} the Neumann trace $\gamma_N$: If $\Gamma_0=\emptyset$, then $\Wscr=H^{1/2}(\partial\Omega)$ and the relation between the two operators is $\gamma_N x=\gamma_\perp\,\Grad x$, for a sufficiently regular $x$, where the equality is in $H^{-1/2}(\partial\Omega)$. The space $H^{-1/2}(\partial\Omega)$ equals $\Wscr'$ in the case where $\Gamma_0=\emptyset$ (which is not the main case of interest to us, see \eqref{eq:bdrpart} below).

Now we include the boundary condition at $\Gamma_0$ into the domain of $\sbm{0&\Div\\\Grad&0}\Hscr$, see \eqref{eq:1orderWave}, by requiring that $\dot w\in H^1_{\Gamma_0}(\Omega)$ instead of the weaker $\dot w\in H^1(\Omega)$ which we motivated above. We can then write \eqref{eq:waveintro} as
\begin{equation}\label{eq:undampedODE}
\left\{\begin{aligned}
		\dot x(t) &= \Afrak\Hscr x(t), \\
		u(t) &= \Bfrak\Hscr x(t), \\
		y(t) &= \Cfrak\Hscr x(t),
\end{aligned}\right.\quad t\geq0,\qquad x(0) = \bbm{\rho\, w_0'\\\Grad w_0},
\end{equation}
where $x(t)=\sbm{\rho\,\dot w(t)\\\Grad w(t)}$ is the state at time $t$, $\Afrak=\sbm{0&\mathrm{div}\\\nabla&0}$, $\Bfrak=\bbm{0&\gamma_\perp}$, and $\Cfrak=\bbm{\gamma_0&0}$, with domains
$$
	\dom\Afrak:=\dom\Bfrak:=\dom\Cfrak:=
		\bbm{ H^1_{\Gamma_0}(\Omega) \\ H^{\mathrm{div}}(\Omega)}
	\subset X_\Hscr,
$$
which is Hilbert when equipped with the graph norm of $\Afrak$. Here $X_\Hscr$ is the state space, $U=\Wscr'$ the input space, and $Y=\Wscr$ the output space. 

In \cite[Thm\ 3.2]{KuZw15} it was shown that \eqref{eq:undampedODE} has the structure of a \emph{boundary triplet} (or abstract \emph{boundary space} in the original terminology of \cite[\S3.1.4]{GoGoBook}). This easily implies that the undamped wave equation is a boundary control system in the sense of Curtain and Zwart \cite[Def.\ 3.3.2]{CuZwbook}:

\begin{definition}\label{def:BCS}
Let the \emph{state space} $X$ and \emph{input space} $U$ be Hilbert spaces, and let $\Ascr: X\supset\dom\Ascr\to X$ and $\Bscr: X\supset\dom\Bscr\to U$ be linear operators with $\dom\Ascr\subset\dom\Bscr$.

The control system $\dot  x(t)=\Ascr\,  x(t)$, $\Bscr\, x(t)=u(t)$, $t\geq0$, $ x(0)= x_0$, is called a \emph{boundary control system (BCS)} if the following conditions are met:
\begin{enumerate} 
\item The operator $A:=\Ascr\big|_{\dom A}$ with domain 
$$
	\dom A:=\dom\Ascr\cap\Ker\Bscr
$$ 
generates a $C_0$-semigroup on $X$ and
\item there exists a $B\in\Lscr(U,X)$ such that $BU\subset\dom\Ascr$, $\Ascr B\in\Lscr(U,X)$, and $\Bscr B=I_U$.
\end{enumerate}
An output equation may be added to the BCS by setting $y(t)=\Cscr x(t)$, where $\Cscr$ is a linear operator defined on $\dom\Cscr\supset\dom\Ascr$ and mapping into some Hilbert \emph{output space} $Y$, with the additional property that $\Cscr B\in\Lscr(U,Y)$. We shall briefly say that $(\Bscr,\Ascr,\Cscr)$ is a BCS on $(U,X,Y)$ if all of the above conditions are met. Finally, we say that a BCS on $(U,X,Y)$ is \emph{(impedance) passive} if the input space $U$ can be identified with the dual $Y'$ of the output space and
$$
	\re\Ipdp{\Ascr x}{x}_X\leq\re\Ipdp{\Bscr x}{\Cscr x}_{Y',Y},\qquad x\in\dom\Ascr.
$$
\end{definition}

For more information on abstract passive BCS, we refer to \cite{MaSt06,MaSt07}. Unlike the setting of Malinen and Staffans, the original definition of Curtain and Zwart does not consider the observation operator $\Cscr$ or passivity, and it is not assumed that $\dom\Ascr$ is a Hilbert space. The robust output regulation theory presented in \S\ref{sec:ROR} below is formulated for the general, abstract systems in Definition \ref{def:BCS}.

We now return to the particular case of the wave equation \eqref{eq:undampedODE}. However, later we shall need to use $L^2(\Gamma_1)$ as both input and output space rather than $\Wscr'$ and $\Wscr$. Fortunately, this can be achieved by restricting $(\Bfrak\Hscr,\Afrak\Hscr,\Cfrak\Hscr)$: Choose the new input space as $\Uscr:=L^2(\Gamma_1)$ and set
$$
	\dom{\widetilde\Afrak}=\set{x\in\Hscr^{-1}\dom\Afrak\bigmid \Bfrak\Hscr x\in L^2(\Gamma_1)}
$$
with the norm given by
$$
	\|x\|_{\dom{\widetilde\Afrak}}^2 := 
	\|\Hscr x\|_{X_\Hscr}^2+\|\Ascr\Hscr x\|_{X_\Hscr}^2+\|\Bscr\Hscr x\|_\Uscr^2.
$$
Furthermore, we define the restrictions
$$
	\widetilde \Afrak:=\Afrak\Hscr\big|_{\dom{\widetilde\Afrak}},\;
	\widetilde \Bfrak:=\Bfrak\Hscr\big|_{\dom{\widetilde\Afrak}},\;
	\widetilde \Cfrak:=\iota\Cfrak\Hscr\big|_{\dom{\widetilde\Afrak}},
$$
where $\iota:\Wscr\to L^2(\Gamma_1)$ is again the (continuous) injection.

\begin{theorem}\label{thm:WaveL2}
The triple $(\widetilde\Bfrak,\widetilde\Afrak,\widetilde\Cfrak)$ is a passive BCS on $(\Uscr,X_\Hscr,\Uscr)$.
\end{theorem}

\begin{proof}
We first note that $\ker(\widetilde\Bfrak)=\Hscr^{-1}\ker(\Bfrak)\subset\dom{\widetilde\Afrak}$ and then we prove that $\widetilde\Afrak\big|_{\ker(\widetilde\Bfrak)}$ generates a unitary group on $X_\Hscr$. It follows from \cite[Cor.\ 3.4]{KuZw15} that $\Afrak\big|_{\ker(\Bfrak)}$ is a skew-adjoint, unbounded operator on $L^2(\Omega)^{n+1}$ 
  and we will show that this implies that $\widetilde\Afrak\big|_{\ker(\widetilde\Bfrak)}$ is skew-adjoint on $X_\Hscr$. Indeed, for an arbitrary fixed $z\in X_\Hscr$, there exists $w\in X_\Hscr$ such that for all $ x\in\ker(\widetilde\Bfrak)=\Hscr^{-1}\ker(\Bfrak)$ we have
 \begin{equation}\label{eq:Afrakadj1}
  \Ipdp{x}{w}_{X_\Hscr} = \Ipdp{\widetilde\Afrak x}{z}_{X_\Hscr} = 
  \Ipdp{\Afrak\Hscr x}{\Hscr z}_{L^2(\Omega)^{n+1}}
\end{equation}
if and only if $\Hscr z\in\dom{\Afrak\big|_{\ker(\Bfrak)}^*}=\Ker\Bfrak$, where the adjoint is computed with respect to the inner product in $L^2(\Omega)^{n+1}$. Hence, $\widetilde\Afrak\big|_{\ker(\widetilde\Bfrak)}$ has the same domain as its adjoint with respect to $X_\Hscr$, and for every $z$ in this common domain, \eqref{eq:Afrakadj1} can be continued as
$$
	\Ipdp{\widetilde\Afrak x}{z}_{X_\Hscr} =
	\Ipdp{\Hscr x}{-\Afrak\Hscr z}_{L^2(\Omega)^{n+1}} =
	\Ipdp{x}{-\widetilde\Afrak z}_{X_\Hscr},
$$
for all $x\in\ker(\widetilde\Bfrak)$. By Stone's theorem, $\widetilde\Afrak$ generates a unitary group on $X_\Hscr$.

As $\gamma_\perp$ maps $H^{\mathrm{div}}(\Omega)$ onto $\Wscr'$, it is clear that $\widetilde\Bfrak$ maps $\dom{\widetilde\Afrak}$ onto $\Uscr$, and thus, $\widetilde B:=\widetilde\Bfrak^{[-1]}\in\Lscr\big(\Uscr,\dom{\widetilde\Afrak}\big)$ has the properties in Definition \ref{def:BCS}.2. Moreover,  the $\Afrak$-boundedness of $\Cfrak$ and the fact that $\Hscr\dom{\widetilde\Afrak}$ is continuously embedded in $\dom{\Afrak}$ imply that $\widetilde\Cfrak \widetilde B\in\Lscr(\Uscr,\Wscr)$. Finally,
$$
	\re\Ipdp{\widetilde\Afrak x}{x}_{X_\Hscr} = 
	\re\Ipdp{\widetilde\Bfrak x}{\widetilde\Cfrak x}_\Uscr,
	\qquad x\in\Dscr \left(\widetilde\Afrak\right),
$$
follows from the following integration by parts formula which was established in the appendix of \cite{KuZw15}, valid for all $h\in H^{\mathrm{div}}(\Omega)$ and $g\in H^1_{\Gamma_0}(\Omega)$:
$$
	\Ipdp{\Div h}{g}_{L^2(\Omega)}+\Ipdp{h}{\Grad g}_{L^2(\Omega)^n} =
	\Ipdp{\gamma_\perp f}{\gamma_0 g}_{\Wscr',\Wscr};
$$
recall that $\Wscr'$ is the dual of $\Wscr$ with pivot space $\Uscr$ and that $\widetilde\Bfrak x\in \Uscr$ for $x\in\dom{\widetilde\Afrak}$.
\end{proof}

\subsection{Exponential stabilization and admissible observation}\label{sec:damper}

The robust controller design in \S\ref{sec:ROR} involves exponential stabilization of the plant with output feedback, and in this section we will comment on this problem for the wave equation~\eqref{eq:waveintro}.
We shall use a special case of a result by Guo and Yao~\cite{GuoYao02} to obtain exponential stabilization using the so-called \emph{multiplier method}. The case where all physical parameters are identity was covered also in \cite{TuWeBook}, see \cite{CuWe06,BaLeRa92} for other related results. 

In order to apply the multiplier method, we assume that the boundary $\partial\Omega$ is of class $C^2$ and that it is partitioned into the reflecting part $\Gamma_0$ and the input/output part $\Gamma_1$ in the following way (see \cite[Chap.\ 7]{TuWeBook} for a longer discussion): Fix $\zeta^0\in\R^n\setminus\overline\Omega$ arbitrarily and define $m(\zeta):=\zeta-\zeta^0$, $\zeta\in\R^n$. We assume that
\begin{equation}\label{eq:bdrpart}
\begin{aligned}
	\Gamma_0 &= \mathrm{int}\,\set{\zeta\in\partial\Omega\mid m(\zeta)\cdot \nu(\zeta)\leq0 }\neq\emptyset \qquad\text{and}\\
	\Gamma_1 &= \set{\zeta\in\partial\Omega\mid m(\zeta)\cdot \nu(\zeta)>0 }\neq\emptyset,
\end{aligned}
\end{equation}
and that the sets $\Gamma_0,\Gamma_1\subset \partial\Omega$ form a partition of the boundary $\partial\Omega$ in the sense that $\overline{\Gamma_0}\cup\overline{\Gamma_1}=\partial\Omega$. In our wave equation, we add a viscous damper 
$u=-b^2y$ on $\Gamma_1$, where
\begin{equation}\label{eq:particulardamper}
	b(\zeta)^2:=m(\zeta)\cdot\nu(\zeta), \qquad\zeta\in\Gamma_1.
\end{equation}
This damper is rigorously interpreted as the following equation in $\Wscr'$: 
$$
	\gamma_\perp T\,\Grad w(t)=-b^2\,\gamma_0\,\dot w(t),\qquad t\geq0.
$$

In order to guarantee \emph{exponential stability}, we do not need to explicitly make the common, but rather restrictive, assumption that $\overline{\Gamma_0}\cap\overline{\Gamma_1}=\emptyset$. However,
combining the assumption that $\partial\Omega$ is of class $C^2$ with the assumption \eqref{eq:dampass} that we need for the admissibility of velocity observation, we unfortunately seem to end up in a situation where necessarily $\overline{\Gamma_0}\cap\overline{\Gamma_1}=\emptyset$.

The total energy associated to a solution $x=\sbm{g\\h}$ of the wave equation in Thm\ \ref{thm:WaveL2} at time $t$ is 
$$
	\displaystyle \frac12\left\|\bbm{g(t)\\h(t)}\right\|_{X_\Hscr}^2 :=
	\frac12\int_\Omega \frac1{\rho(\zeta)}\,g(\zeta,t)^2+h(\zeta,t)^* T(\zeta) h(\zeta,t)\ud\zeta,
$$
representing the sum of kinetic and potential energy.

\begin{theorem}\label{thm:waveexpstab}
Assume that $\rho$ and $T$ are constant, that $\Omega\subset\R^n$ is a bounded $C^2$-domain with $n\leq 3$, and that $\Gamma_k$ satisfy \eqref{eq:bdrpart}. Then there exist $c>1$ and $\omega>0$, such that all $\sbm{g\\h}\in C^1(\rplus;X_\Hscr)$ with $\ddt\sbm{g(t)\\h(t)}=\Ascr\Hscr\sbm{g(t)\\h(t)}$ and $\gamma_\perp T\,h(t)=-(m\cdot \nu)\,\gamma_0\,g(t)$ for $t\geq0$, and $h(0)\in \Grad H^1_{\Gamma_0}(\Omega)$, satisfy
\begin{equation}\label{eq:expstab}
	\left\|\bbm{g(t)\\h(t)}\right\|_{X_\Hscr}^2
	\leq c\,e^{-\omega t}\,\left\|\bbm{g(0)\\h(0)}\right\|_{X_\Hscr}^2,
	\qquad t\geq0.
\end{equation}
\end{theorem}

\begin{proof}
Let $\sbm{g\\h}$ have the properties in the statement and let $\eta\in H^1_{\Gamma_0}(\Omega)$ be such that $\Grad\eta=h(0)$. Setting 
\begin{equation}\label{eq:wavesoldef}
	w(t):=\eta+\frac1\rho\int_0^t g(s)\ud s,\qquad t\geq0,
\end{equation}
we get that $\dot w(t)=g(t)/\rho$ and $\Grad w(t)=h(t)$ for all $t\geq0$. Moreover, $w$ is a classical solution of the wave equation since 
\begin{equation}\label{eq:waveGY}
	\ddot w(t)=\Div\left(\frac T\rho\,\Grad w\right)(t),\qquad t\geq0,
\end{equation}
with the left-hand side in $C\big(\rplus;L^2(\Omega)\big)$. Note that the constant matrix $T/\rho$ is positive definite and hence invertible. 

In \cite[Ex.\ 3.1]{Yao99}, a Riemannian manifold $(\R^n,g)$ is associated to \eqref{eq:waveGY}, and it is concluded that the vector field $H:=\sum_{k=1}^n (\zeta_k-\zeta_k^0)\,\partial/\partial \xi_k$ on this manifold satisfies the condition \cite[(3.2)]{GuoYao02} with $a=1$ (here $\zeta_k$ is coordinate number $k$ of $\zeta$). We further observe that $w(t)\in H^1_{\Gamma_0}(\Omega)$ and $\gamma_\perp T\Grad w(t)=-(m\cdot \nu)\,\gamma_0\,\dot w(t)$ for all $t\geq0$, while $w(0)\in H^1_{\Gamma_0}(\Omega)$, and $\dot w(0)\in L^2(\Omega)$.
By \cite[Thm\ 1]{GuoYao02}, we have \eqref{eq:expstab}.
\end{proof}

In general, a solution $w$ of \eqref{eq:1orderWave} is only required to be \emph{constant} on $\Gamma_0$. The condition $h(0)\in\Grad H^1_{\Gamma_0}(\Omega)$ corresponds to the initial condition $w(0)\in H^1_{\Gamma_0}(\Omega)$ via \eqref{eq:wavesoldef}, and this implies the stronger statement that $w$ is constantly equal to \emph{zero} on $\Gamma_0$. This is one way to guarantee that the potential energy decays to zero.

Returning to the case of the general BCS, we will replace the multiplication by $-m\cdot\nu$ on $L^2(\Gamma_1)$ by an admissible output feedback operator $Q\in\Lscr(\Yscr,\Uscr)$ which stabilizes the given BCS exponentially: Let $(\Bscr,\Ascr,\Cscr)$ be a BCS on $(U,X,Y)$. We call $Q\in\Lscr(Y,U)$ an \emph{admissible (static output) feedback operator} for $(\Bscr,\Ascr,\Cscr)$ if $(\Bscr+Q\Cscr,\Ascr,\Cscr)$ is a also a BCS. Moreover, let the Hilbert spaces $Y$ and $Y'$ be duals with some pivot Hilbert space $\widetilde U$, and let $Q\in\Lscr(Y,Y')$. We say that $Q$ is \emph{uniformly accretive} if there exists some $\delta>0$ such that
$$
	\re\Ipdp{Qy}{y}_{Y',Y}\geq \delta\,\|y\|^2_{\widetilde U},	\qquad y\in Y.
$$ 

By an \emph{admissible observation operator} for a $C_0$-semigroup $\T$ on $X$ with generator $A$, we mean a linear operator $C\in\Lscr(\dom A,Y)$ for which there exist some $\tau>0$ and $K_\tau\geq0$ such that
  \begin{equation}\label{eq:adm}
    \int_0^\tau \|C\,\T(t)\,x\|^2_Y\ud t\leq K_\tau^2\,\|x\|^2_X,
    \qquad \forall x\in\dom A.
  \end{equation}
If \eqref{eq:adm} holds for some $\tau>0$ and $K_\tau\geq0$, then for every $\tau>0$ it is possible to choose a $K_\tau\geq 0$ such that~\eqref{eq:adm} holds. The observation operator is \emph{infinite-time admissible} if \eqref{eq:adm} holds for all $\tau>0$ with $K_\tau$ replaced by some bound $K$ which is independent of $\tau$. In particular, if the semigroup $\T$ is exponentially stable, then every admissible observation operator is infinite-time admissible~\cite[Prop. 4.3.3]{TuWeBook}.

\begin{proposition}\label{prop:adm}
Let $(\Bscr,\Ascr,\Cscr)$ be a \emph{passive} BCS on $(Y',X,Y)$ and let $Q\in\Lscr(Y,Y')$ be a uniformly accretive, admissible output feedback operator for $(\Bscr,\Ascr,\Cscr)$. The resulting BCS $(\Bscr+Q\Cscr,\Ascr,\Cscr)$ is also passive and we denote its associated semigroup by $\T_Q$. The observation operator $\Cscr$, interpreted as an operator mapping into the pivot space $\widetilde Y$ rather than into $Y$, is infinite-time admissible for $\T_Q$. 
\end{proposition}

\begin{proof}
By the definitions of admissible feedback operator and BCS, it follows that $(\Bscr+Q\Cscr,\Ascr,\Cscr)$ is a BCS on $(Y',X,\widetilde Y)$, and by definition the generator of $\T_Q$ is $A_Q:=\Ascr\big|_{\Ker{\Bscr+Q\Cscr}}$. 

For a fixed $ x_0\in\dom{A_Q}$, the associated state trajectory $ x(t)=\T_Q (t)\, x_0$ stays in $\dom{A_Q}$, and by the assumed passivity, for all $t\geq0$ we have
$$
  	\re\Ipdp{\Ascr  x(t)}{ x(t)}_X \leq \re\Ipdp{\Bscr  x(t)}{\Cscr  x(t)}_{Y',Y}
  		= -\re\Ipdp{Q\Cscr  x(t)}{\Cscr  x(t)}_{Y',Y}.
$$
Multiplying this by 2 and integrating over $[0,\tau]$, we get
$$
	\| x(\tau)\|_X^2-\| x(0)\|_X^2 = \int_0^\tau 2\re\Ipdp{A_Q  x(t)}{ x(t)}_X\ud t 
	\leq -2\delta\int_0^\tau \|\Cscr x(t)\|^2_{\widetilde Y}\ud t.
$$
Letting $\tau\to+\infty$, we obtain that $\Cscr$ is infinite-time admissible, since
$$
	\int_0^\infty \|\Cscr \T_Q(t) x_0\|^2_{\widetilde Y}\ud t \leq \frac{1}{2\delta}\| x_0\|_X^2,
		\quad x_0\in\dom{A_Q}.
$$
\end{proof}

We end the section by discussing the wave system as an example for the above abstract definitions. It is clear that the multiplication by $b^2=m\cdot\nu$ in \eqref{eq:particulardamper} is a bounded operator on $L^2(\Gamma_1)$, and hence it is also in $\Lscr(\Wscr,\Wscr')$ and it is  uniformly accretive if \eqref{eq:dampass} holds. Furthermore, multiplication by $b^2$ is an admissible feedback operator for the wave system in \eqref{eq:undampedODE} and for its restriction in Thm \ref{thm:WaveL2}. Indeed, $\ker(\Bfrak\Hscr+b^2\,\Cfrak\Hscr)=\ker(\widetilde\Bfrak+b^2\,\widetilde\Cfrak)\subset\dom{\widetilde\Afrak}$, by \cite[Thm\ 3.5]{KuZw15} the operator $\Afrak\Hscr\big|_{\Ker{\Bfrak\Hscr+b^2\,\Cfrak\Hscr}}=\widetilde\Afrak\big|_{\ker(\widetilde\Bfrak+b^2\,\widetilde\Cfrak)}$ generates a  contraction semigroup on $X_\Hscr$, and the operators
$$
	\Bfrak\Hscr+b^2\Cfrak\Hscr=\bbm{b^2\gamma_0&\gamma_\perp}\Hscr
	\qquad\text{and}\qquad \widetilde\Bfrak+b^2\,\widetilde\Cfrak
$$ 
are continuous and surjective; hence they have right-inverses with the properties required in Definition \ref{def:BCS}.

\section{The plant, the controller, and the exosystem} \label{sec:concontr} 

In the next section, we solve the robust output regulation problem for a general BCS $(\Bscr,\Ascr,\Cscr)$ on the Hilbert spaces $(U,X,Y)$; the system is not necessarily related to the wave equation. In the following we assume that the whole boundary $\partial \Omega$ is accessible via $\Bscr$ and $R_1,R_2$ are arbitrary restrictions to certain parts of $\partial \Omega$. We first add an external disturbance $w$ to the BCS, thus obtaining the plant
\begin{equation}\label{eq:plant}
	\left\{\begin{aligned} 
		\dot x(t)&=\Ascr x(t), \qquad x(0)=x_0,\\
		\Bscr x(t)&=R_1u(t)+R_2w(t),   \qquad t\geq0, \\
		\Cscr x(t)&=y(t),
	\end{aligned}\right.
\end{equation}
where $u$ and $w$ may act on different parts of the boundary depending on $R_1$ and $R_2$.

In what follows, $Q$ is such that $R_1Q$ is an admissible static output feedback operator for \eqref{eq:plant} such that the semigroup $\T_s$ generated by $A_s:=\Ascr\big|_{\dom\Ascr\cap\Ker{\Bscr+R_1Q\Cscr}}$ is exponentially stable and $\Cscr$ is an admissible observation operator for $\T_s$ (here the subscript 's' stands for "stabilized plant'').

We will connect the plant to the dynamic controller
\begin{equation}\label{eq:controller}
	\left\{\begin{aligned}         
		\dot z(t) &= \Gscr_1 z(t)+\Gscr_2(y(t)-y_{ref}(t)), \quad z(0) = z_0 \\
		u(t) &= Kz(t)-Q(y(t)-y_{ref}(t)),\qquad t\geq0,
	\end{aligned}\right.
\end{equation}
where $y_{ref}$ is an external reference signal and the state space $Z$ of the controller is a Hilbert space, but $\Gscr_1\in\Lscr(Z)$ is bounded. Moreover, we assume that $\Gscr_2\in\Lscr(Y,Z)$, $K\in\Lscr(Z,U)$ and $Q\in\Lscr(Y,U)$. The disturbance signal $w$ and the reference signal $y_{ref}$ are assumed to be generated by an exosystem
\begin{equation}\label{eq:exo}
	\left\{\begin{aligned} 
	\dot v(t)&=S v(t), \quad v(0)=v_0, \\
	 w(t)&=Ev(t), \qquad t\geq0,  \\
	y_{ref}(t)&=-Fv(t), 
\end{aligned}\right.  
\end{equation}
which is a linear system on a finite-dimensional space $W = \mathbb{C}^q$, $q \in \mathbb{N}$. We assume that $S = \operatorname{diag}(i\omega_1, i\omega_2,\ldots, i\omega_q)$ with $\omega_i \neq \omega_j$ for $i \neq j$, $E \in \Lscr (W, U)$ and $F \in \Lscr (W, Y)$.

Setting $u$ and $y$ equal in \eqref{eq:plant} and \eqref{eq:controller}, and using \eqref{eq:exo}, we obtain 
\begin{equation}\label{eq:closed-loop}
	\left\{\begin{aligned} 
		\ddt\bbm{x\\z}&=\bbm{\Ascr&0\\\Gscr_2\Cscr&\Gscr_1}\bbm{x\\z}
			+\bbm{0\\\Gscr_2F}v, \\
		(R_2E - R_1QF)v &=\bbm{\Bscr+R_1Q\Cscr & -R_1K}\bbm{x\\z} ,\\
		e &=\bbm{\Cscr&0} \bbm{x\\z} + Fv,
	\end{aligned}\right.
\end{equation}
where we chose the regulation error $e(t) =: y(t) - y_{ref}(t)$ as the output and the state-space is $X_e := X \times Z$. This system is no longer a BCS and we now proceed to write it in the standard input/state/output form. First we observe that we may interpret the feedthrough $Q$ of the controller as a part of the plant without changing \eqref{eq:closed-loop}. This amounts to pre-stabilizing the plant via replacing the input equation of \eqref{eq:plant} by $(\Bscr+R_1Q\Cscr)x(t) =R_1u(t)+(R_2E - R_1QF)v(t)$ and simultaneously removing the term $-Q(y(t)-y_{ref}(t))$ from the output equation of \eqref{eq:controller}.

As $R_1Q$ is assumed to be an admissible feedback operator, the pre-stabilized plant $(\Bscr+R_1Q\Cscr,\Ascr,\Cscr)$ is a BCS and by Def.\ \ref{def:BCS}.2, we can choose a right inverse $B_s\in\Lscr(U,X)$ of $\Bscr +R_1Q\Cscr$ such that 
\begin{equation} \label{eq:Bprop}
 B_sR_1U\subset\dom\Ascr, \; \Ascr B_sR_1\in\Lscr(U,X), \; \Cscr B_sR_1\in\Lscr(U,Y). 
 \end{equation}
In order to present the transfer function of $(\Bscr+R_1Q\Cscr,\Ascr,\Cscr)$, consider the auxiliary function
$$
P_0(\lambda) := \Cscr(\lambda-A_s)^{-1}(\Ascr B_s-\lambda B_s)+\Cscr B_s,\quad
	\lambda\in\res{A_s}.
$$
Now, define the transfer function by
\begin{equation} \label{eq:transf}
	P_s(\lambda) :=P_0(\lambda)R_1, \qquad \lambda\in\res{A_s}.
\end{equation}
The auxiliary function $P_0$ becomes useful later on in describing the mapping from $v$ to $y$.

Now let $\sbm{x\\z}$ be a classical state trajectory of \eqref{eq:closed-loop}, i.e., $\sbm{x\\z}\in C^1(\rplus;X_e)$, $\Gscr_2 y_{ref}\in C(\rplus;Z)$, $(\Bscr+R_1Q\Cscr)x\in C(\rplus;U)$, $w\in C^1(\rplus;U)$, and the first two lines of \eqref{eq:closed-loop} hold for all $t\geq0$. Next introduce a new state variable for \eqref{eq:closed-loop} by
$$
	x_e:=\bbm{1&-B_sR_1 K\\0&1}\bbm{x\\z}-\bbm{B_sE_s v\\0} \in C^1(\rplus;X_e),
$$
where we denote $E_s := R_2E-R_1QF$ for brevity. This transformation can be inverted as 
\begin{equation}\label{eq:statetransfinv}
	\bbm{x\\z}:=\bbm{1&B_s R_1K\\0&1}x_e+\bbm{B_sE_sv\\0}.
\end{equation}
Differentiating $x_e$ and using the first line of \eqref{eq:closed-loop}, we get
$$
\begin{aligned}
	\dot x_e &= \bbm{\Ascr-B_sR_1K\Gscr_2\Cscr&\Ascr B_sR_1K-B_sR_1K\tilde \Gscr_1  \\ 
		\Gscr_2\Cscr & \tilde\Gscr_1} x_e \\
	&\qquad + \bbm{\Ascr B_s E_s - B_sE_sS - B_sR_1K\Gscr_2(\Cscr B_sE_s + F) \\ \Gscr_2(\Cscr B_sE_s + F)}v,
\end{aligned}
$$
where we denote $\tilde\Gscr_1 := \Gscr_1+\Gscr_2\Cscr B_sR_1K$ for brevity.

With the new state variable, the input equation of \eqref{eq:closed-loop} becomes
$$	
	E_s v=\bbm{\Bscr+R_1Q\Cscr&-K}\left(x_e+\bbm{B_s\\0}(R_1Kz+E_sv)\right)
$$
which simplifies to $x_e\in\Ker{\bbm{\Bscr+R_1Q\Cscr&0}}$. Hence recalling that $A_s=\Ascr\big|_{\dom\Ascr\cap\Ker{\Bscr+R_1Q\Cscr}}$ and defining
\begin{equation}\label{eq:AeDef}
\begin{split}
	A_e &:= \bbm{A_s-B_sR_1K\Gscr_2\Cscr&\Ascr B_sR_1K-B_sR_1K\tilde\Gscr_1 \\ 
		\Gscr_2\Cscr & \tilde\Gscr_1}\bigg|_{\dom{A_e}}, \\
	\dom{A_e} &:= \Ker{\Bscr+R_1Q\Cscr}\times Z,
\end{split}
\end{equation}
we get that every classical solution of \eqref{eq:closed-loop} satisfies $x_e(t)\in\dom{A_e}$ for all $t\geq0$ and $\dot{ x}_e = A_e  x_e + B_ev$,
where the control  operator $B_e\in\Lscr(W, X_e)$ is
$$
	B_e:=\bbm{\Ascr B_s E_s - B_sE_sS - B_sR_1K\Gscr_2(\Cscr B_sE_s + F) \\ \Gscr_2(\Cscr B_sE_s + F)}.
$$
Finally, using \eqref{eq:statetransfinv} the output for \eqref{eq:closed-loop} becomes
$$	
	e = \bbm{\Cscr & \Cscr B_sR_1K} x_e + (\Cscr B_sE_s+F)v.
$$
Thus, the closed-loop system  is of the form
\begin{equation}\label{eq:clloopsys} 
	\left\{ \begin{aligned}
		\dot x_e & = A_e x_e + B_e v, \\ 
		e & = C_e x_e + D_e v,
	\end{aligned}\right.
\end{equation}
where
$$
\begin{aligned}
	C_e &:= \bbm{\Cscr & \Cscr B_sR_1K},\qquad\dom{C_e}:=\bbm{\dom\Cscr\\Z},
		\qquad\text{and} \\
	D_e &:= \Cscr B_sE_s+F\in\Lscr(W, Y).
\end{aligned}
$$
We denote the transfer function of \eqref{eq:clloopsys} from $v$ to $e$ with 
$$
P_e(\lambda) = C_e(\lambda-A_e)^{-1}B_e + D_e.
$$

The above calculations show that every classical solution of \eqref{eq:closed-loop} with $v\in C(\rplus;W)$ is also a classical solution of \eqref{eq:clloopsys}. Conversely, assume that $x_e\in C^1(\rplus;X_e)$ with $x_e(t)\in\dom{A_e}$, $v\in C(\rplus;W)$ and \eqref{eq:clloopsys} holds on $\rplus$. Then $v$, $\sbm{x\\z}$ in \eqref{eq:statetransfinv} and $e$ satisfy \eqref{eq:closed-loop}. We conclude that \eqref{eq:closed-loop} and \eqref{eq:clloopsys} are equivalent systems in the sense that they have the same classical solutions. 

The following result forms the basis for the output regulation theory in the next section. Note that we do not assume that the original plant \eqref{eq:plant} is well-posed or regular, but the closed-loop system \eqref{eq:clloopsys} nevertheless has these properties. 

\begin{theorem}\label{thm:CLsemigroup}
The operator $A_e$ in \eqref{eq:AeDef} generates a $C_0$-semigroup $\T_e$ on $X_e$ and $C_e$ is an admissible observation operator for $\T_e$.  The closed-loop system \eqref{eq:clloopsys} is well-posed and regular such that $P_e(\lambda) \to D_e$ as $\operatorname{Re}\lambda \to \infty$.
\end{theorem}

\begin{proof}
We begin by splitting $A_e=A_1+A_2+A_3$, where
$$
\begin{aligned}
	A_1 &= \bbm{ A_s & 0 \\ 0 & \Gscr_1 }, \qquad \dom{A_1}=\dom{ A_e}, \\
	A_2 &= \bbm{ -B_sR_1K\Gscr_2\Cscr&0 \\ \Gscr_2\Cscr & 0}, \qquad \dom{A_2}=\dom{ A_e}, \\
	A_3 &= \bbm{ 0&\Ascr B_sR_1K-B_sR_1K(\Gscr_1+\Gscr_2\Cscr B_sR_1K) \\ 
		0 & \Gscr_2\Cscr B_sR_1K}, \\ & \qquad \dom{A_3}=X_e.
\end{aligned}
$$
Here $A_1$ generates a $C_0$-semigroup $\T_1$ on $X_e$. The operator $A_2$ can be factored as
$$
	A_2 = \bbm{-B_sR_1K\Gscr_2\\\Gscr_2}\bbm{\Cscr&0},
$$
where the first factor is bounded from $Y$ into $X_e$. Our assumption that $\Cscr$ is admissible for $\T_s$  implies that $\bbm{\Cscr&0}:X_e\supset\dom{A_e}\to Y$ is an admissible observation operator for $\T_1$, and by \cite[Thm 5.4.2]{TuWeBook}, $A_1+A_2$ generates a $C_0$-semigroup $\T_2$ on $X_e$ and $\bbm{\Cscr&0}$ is admissible for $\T_2$. Since $A_3$ is bounded, $A_e$ generates a $C_0$-semigroup by \cite[Thm 5.4.2]{TuWeBook} and due to the boundedness of $\Cscr B_sR_1K$, $C_e$ is admissible for $\T_e$. As in addition $B_e$ and $D_e$ are bounded, the well-posedness and regularity of the closed-loop system follow immediately from \cite[Thm\ 4.3.7]{TuWeBook}
\end{proof}

\section{Output regulation}\label{sec:ROR}

We begin this section by presenting the three output regulation problems considered in this paper. The structure for the remainder of this section will be presented after the problem definitions.

\medskip

\noindent {\bf The Output Regulation Problem.} For a given plant \eqref{eq:plant}, choose the controller $(\Gscr_1, \Gscr_2, K, Q)$ in \eqref{eq:controller} in such a way that the following are satisfied:
\begin{enumerate}
	\item The closed-loop system generated by $A_e$ is exponentially stable.
	\item For all initial states $ x_{e0} \in X_e$ and $v_0 \in W$ the regulation error satisfies $e^{\alpha \cdot} e(\cdot) \in L^2([0,\infty); Y)$ for some $\alpha > 0$ independent of $ x_{e0} \in X_e$ and $v_0 \in W$.
\end{enumerate}
Furthermore, if the controller solves the output regulation problem despite perturbations in the parameters of the plant or the exosystem, then we say that the controller solves the \emph{robust} output regulation problem with respect to this class of perturbations. To make this precise, we first define the class of admissible perturbations:

\begin{definition}
A quintuple $(\Ascr', \Bscr', \Cscr', E', F')$ of linear operators belongs to the \emph{class $\Oscr$ of admissible perturbations} if it has the following properties: 
\begin{enumerate}
\item The triple $(\Bscr'+R_1Q\Cscr',\Ascr',\Cscr')$ is a BCS on $(U,X,Y)$.
\item The observation operator $\Cscr'$ is admissible for the semigroup generated by $A_s':=\Ascr'\big|_{\Ker{\Bscr'+R_1Q\Cscr'}}$.
\item The eigenvalues of $S$ are in the resolvent set of the perturbed pre-stabilized plant, i.e., $\{i\omega_k\}_{k=1}^q \subset \rho(A_s')$.
\item $E' \in \Lscr (W, U)$ and $F' \in \Lscr(W,Y)$.
\end{enumerate}
\label{ass:standing}
\end{definition}

In the above definition it would appear that the class $\Oscr$ of perturbations depends on $Q$. However, as $Q$ only contributes to stabilizing the plant, we have much more freedom choosing $Q$ than choosing the other controller parameters (as seen later on). For example, in the wave equation considered in Section \ref{sec:firstord}, any uniformly accretive operator can be chosen as $Q$. Therefore, in Definition \ref{ass:standing}, one could think of $Q$ being chosen such that the class $\Oscr$ is as large as possible. Moreover, if $(\Ascr', \Bscr', \Cscr', E', F')\in\Oscr$ then the transfer function \eqref{eq:transf} of the triple $(\Bscr'+R_1Q\Cscr',\Ascr',\Cscr')$ is well-defined and bounded at the frequencies of the exosystem.

We make the natural assumption that the unperturbed system is in class $\Oscr$ as well, that is, $(\Ascr, \Bscr, \Cscr, E, F) \in \Oscr$. Note that this does not include the assumption that the semigroup generated by $A_s$ is exponentially stable. Further note that even though $(\Bscr, \Ascr, \Cscr)$ is assumed to be a BCS, that is not required from $(\Bscr', \Ascr, \Cscr')$ but only from $(\Bscr' + R_1Q\Cscr', \Ascr', \Cscr')$.

From Definition \ref{ass:standing} it follows that the perturbed closed-loop system is well-posed and regular. Please note that while no perturbations are allowed in the eigenvalues of the generator $S$ of the exosystem or in the controller parameter $\Gscr_1$, the parameters $(\Gscr_2, K, Q)$ would in fact allow certain bounded perturbations. We will comment on this more thoroughly in Remark \ref{rem:contr_per}.

\medskip

\noindent {\bf The Robust Output Regulation Problem.} For a given plant, choose the controller $(\Gscr_1, \Gscr_2, K, Q)$ in such a way that the following are satisfied:
\begin{enumerate}
	\item The controller $(\Gscr_1, \Gscr_2, K, Q)$ solves the output regulation problem.
	\item If the operators $(\Ascr, \Bscr, \Cscr, E, F)$ are perturbed to $(\Ascr', \Bscr', \Cscr', E', F') \in \Oscr$ in such a way that the closed-loop system remains exponentially stable, then for all initial states $ x_{e0} \in X_e$ and $v_0 \in W$ the regulation error satisfies $e^{\alpha'\cdot} e(\cdot) \in L^2([0,\infty); Y)$ for some $\alpha' > 0$ independent of $ x_{e0} \in X_e$ and $v_0 \in W$.
\end{enumerate}

In Section \ref{ROR:arrc}, we will construct a controller that solves the robust output regulation problem \emph{approximately}. That is, the regulation error does not decay asymptotically to zero but can be made small. For this purpose, we introduce the following new control problem:

\medskip

\noindent {\bf The Approximate Robust Output Regulation Problem.} Let $\delta > 0$ be given. Choose the controller $(\Gscr_1, \Gscr_2, K, Q)$ in such a way that the following are satisfied:
\begin{enumerate}
\item The closed-loop system generated by $A_e$ is exponentially stable.
\item For all initial states $ x_{e0} \in X_e$ and $v_0 \in W$ the regulation error satisfies 
$$
	\int_t^{t+1 }\|e(s)\|^2\ud s \leq Me^{-\alpha t}(\|x_{e0}\|^2 + \|v_0\|^2) + \delta\|v_0\|^2
$$ 
for some $M,\alpha > 0$ independent of $x_{e0} \in X_e, v_0 \in W.$
\item If the operators $(\Ascr, \Bscr, \Cscr, E, F)$ are perturbed to $(\Ascr', \Bscr', \Cscr', E', F') \in \Oscr$ in such a way that the closed-loop system remains exponentially stable, then there exists a ¡$\delta' > 0$ such that for all initial states $ x_{e0} \in X_e$ and $v_0 \in W$ the regulation error satisfies 
$$
	\int_t^{t+1 }\|e(s)\|^2\ud s \leq M'e^{-\alpha' t}(\|x_{e0}\|^2 + \|v_0\|^2) + \delta'\|v_0\|^2
$$ 
for some $M',\alpha' > 0$ independent of $x_{e0}, v_0$.
\end{enumerate}

\begin{remark} \label{rem:arorp}
The approximate robust output regulation problem formulation implies that, in the absence of perturbations, the asymptotic regulation error must be smaller than $\delta\|v_0\|^2$ for any given (or in practice chosen) $\delta > 0$. However, when perturbations are present, the asymptotic regulation error is merely bounded by $\delta'\|v_0\|^2$. For details, see Theorem \ref{thm:approxcontrol}, \eqref{eq:appr_error}--\eqref{eq:delta} and the discussion therein.
\end{remark}

Now that we have presented the different output regulation problems to be considered, the structure of the remaining section is as follows. Before proceeding to constructing the controllers, we will present two auxiliary results to be used throughout the remainder of this section. In \S\ref{ROR:nrrc} we present a regulating controller without the robustness requirement, in \S\ref{ROR:ROR} we present the internal model principle for boundary control systems, in \S\ref{ROR:arrc} we present an approximate robust controller, and finally in \S\ref{ROR:rrc} we present a precise robust controller.

The following auxiliary result is a consequence of \cite[Thm 4.1]{PauPoh14} under the assumption that the closed-loop system \eqref{eq:clloopsys} is a regular linear system. The result states that the solvability of the {\itshape regulator equations}
\begin{subequations}
\begin{align}
	\Sigma S & = A_e\Sigma + B_e \label{eq:regeq1} \\
	0 & = C_e\Sigma + D_e \label{eq:regeq2}
\end{align}
\label{eq:regeq}%
\end{subequations}
is equivalent to the solvability of the output regulation problem. The result of \cite[Thm 4.1]{PauPoh14} essentially follows from \cite[Lem. 4.3]{PauPoh14} by which the regulation error can be written as 
$$
	e(t) = C_e\mathbb{T}_e(t)(x_{e0} - \Sigma v_0) + (C_e\Sigma + D_e)v(t),
$$
where the first part decays to zero at an exponential rate provided that $\mathbb{T}_e$ is exponentially stable, $C_e$ is an admissible observation operator for $\mathbb{T}_e$ and $\Sigma$ is the solution of \eqref{eq:regeq1}.

\begin{theorem}
Assume that the closed-loop system is regular and exponentially stabilized by a controller $(\Gscr_1, \Gscr_2, K, Q)$. Then the controller solves the output regulation problem if and only if the regulator equations \eqref{eq:regeq} have a solution $\Sigma \in \Lscr(W,X_e)$. The solution $\Sigma$ is unique when it exists.
\label{thm:PauPoh14_4.1}
\begin{proof}
We first note that the feedthrough term $-Qe(t)$ in the controller is not part of the controller in \cite[Thm 4.1]{PauPoh14}. However, as in \eqref{eq:closed-loop} we can interpret the feedthrough $Q$ as a part of the plant \eqref{eq:plant} and simultaneously remove it from the controller \eqref{eq:controller}, so that the input equation becomes $(\Bscr + R_1Q\Cscr)x(t) = R_1u(t) + R_1Qy_{ref}(t) + R_2w(t)$. The closed-loop system is unaffected by this algebraic trick, and hence, we may continue with a pre-stabilized plant and the same controller structure as in \cite[Thm 4.1]{PauPoh14}.

Now the result follows from \cite[Thm 4.1]{PauPoh14} as an exponentially stable semigroup is also strongly stable, and for $A_e$ being the generator of an exponentially stable semigroup and $\sigma(S) \subset i\mathbb{R}$ the Sylvester equation $\Sigma S =  A_e \Sigma +  B_e$ always has a unique solution $\Sigma \in \Lscr(W,X_e)$ by \cite[Cor. 8]{Pho91}. Furthermore, the exponential decay of the regulation error follows from the assumed exponential stability of the closed-loop system.
\end{proof}
\end{theorem}

Theorem \ref{thm:PauPoh14_4.1} assumes that the controller exponentially stabilizes the closed-loop system. We will therefore need to show that the controllers we present in Proposition \ref{thm:nrrc}, Theorem \ref{thm:approxcontrol} and Corollary \ref{cor:robregcontr} have this property. For this, we present the following tool which uses the notation of \S\ref{sec:concontr}. Here we need to assume that there exists an operator $Q$ as described in the following:

\begin{lemma}
Let $Z = Y_N^q$, where $Y_N$ is equal to $\C$ or a closed subspace of $Y$. Choose the controller parameter $Q\in\Lscr(Y,U)$ such that the semigroup $\T_s$ generated by $A_s$ is exponentially stable and $\Cscr$ is an admissible observation operator for $\T_s$. Choose the remaining parameters as 
$$ 
\begin{aligned}
	\Gscr_1& = \operatorname{diag}\left(i\omega_1I,i\omega_2I, \ldots, i\omega_qI\right) \in \Lscr(Z), \\
	K & = \epsilon K_0  = \epsilon [K_0^1, K_0^2, \ldots, K_0^q] \in \Lscr(Z,U), \\
	\Gscr_2 & = (\Gscr_2^kP_N)_{k=1}^q \in \Lscr(Y,Z),
\end{aligned}
$$
where $I$ is the identity in $Y_N$, and $P_N$ is a projection onto $Y_N$ in $Y$ if $Y_N\subset Y$ or the identity on $Y$ otherwise. Additionally, assume that $\Gscr_2^k$ and $K_0^k$ satisfy $\sigma(\Gscr_2^kP_NP_s(i\omega_k)K_0^k) \subset \mathbb{C}_-$ for all $k \in \{1,2,\ldots,q\}$. 

Then there exits an $\epsilon^* > 0$ such that the closed-loop system \eqref{eq:clloopsys} is exponentially stable for all $0 < \epsilon < \epsilon^*$.
\begin{proof}
Define the operator $H = (H_1,H_2,\ldots,H_q) \in \Lscr(Z,X)$ by choosing
$$
	H_k := (i\omega_k - A_s)^{-1}(\Ascr B_s - i\omega_k B_s)R_1K_0^k 
$$
for all $k \in \{1,2,\ldots,q\}$. By the choice of $H_k$ we have $(i\omega_k - A_s)H_k = \Ascr B_s R_1K_0^k - i\omega_k B_sR_1 K_0^k$, i.e., $H_k i\omega_k = A_s H_k + \Ascr B_s R_1K_0^k - B_sR_1 K_0^k i\omega_k$, and thus, $H\Gscr_1 = A_s H + \Ascr B_sR_1 K_0 - B_s R_1K_0\Gscr_1$ due to the diagonal structure of $\Gscr_1$. Define
$$
	R = \left[ \begin{array}{rr} -1 & \epsilon H \\ 0 & 1 
\end{array} \right] = R^{-1} \in \Lscr(X_e)
$$
and denote $\hat{A}_e = RA_e R^{-1}$. Note that as $\Rscr(H) \subset \Nscr(\Bscr + R_1Q\Cscr)$, it follows that $\Dscr(\hat{A}_e) = \Dscr(A_e)$. Using the above identity we can write $\hat{A}_e$ as
$$
\hat{A}_e =  \bbm{A_s  - \epsilon \tilde H\Gscr_2 \Cscr & 0 \\
	-\Gscr_2 \Cscr & \Gscr_1 + \epsilon \Gscr_2 \Cscr\tilde H} + \epsilon^2 \bbm{
	0 & \tilde H\Gscr_2 \Cscr\tilde H \\ 0 & 0}.
$$
where we denote $\tilde H := H + B_sR_1K_0$ for brevity.

In the remaining part of the proof we apply the Gearhart-Pr\"uss-Greiner Theorem in \cite[Thm V.1.11]{EnNaBook}. More precisely, we will show that the resolvent of $\hat A_e$ is uniformly bounded on the closed right-half plane. We first note that since $\Cscr$ is admissible for $\T_s$ which is exponentially stable, we have by \cite[Thm 4.3.7]{TuWeBook} that $\Cscr(\lambda - A_s)^{-1}$ is uniformly bounded for all $\lambda\in \overline\C_+$. Thus, as $\tilde H\Gscr_2$ is bounded, there exists an $M_0 > 0$ such that $\|\tilde H\Gscr_2\Cscr(\lambda - A_s)^{-1}\| \leq M_0$, and for $0 < \epsilon < M_0^{-1}$ a Neumann series expansion implies that $1 + \epsilon\tilde H\Gscr_2\Cscr(\lambda - A_s)^{-1}$ is invertible. Thus, we obtain that
$$
(\lambda - A_s + \epsilon\tilde H\Gscr_2\Cscr)^{-1} = (\lambda - A_s)^{-1}(1 + \epsilon\tilde H\Gscr_2\Cscr(\lambda - A_s)^{-1})^{-1}
$$
is uniformly bounded in the right half plane. Hence, the semigroup generated by $A_s - \epsilon\tilde H\Gscr_2\Cscr$ is exponentially stable by \cite[Thm V.1.11]{EnNaBook}.

Note that by the choice of $H_k$ we have
$$
\begin{aligned}
\Cscr(H_k + B_sR_1 K_0^k) & = \Cscr (i\omega_k - A_s)^{-1}(\Ascr B_s - i\omega_k B_s)R_1K_0^k + \Cscr B_s R_1K_0^k \\ & = P_s(i\omega_k)K_0^k,
\end{aligned}
$$
and thus  $\sigma(\Gscr_2^kP_N\Cscr(H_k + B_sR_1K_0^k)) \subset \mathbb{C}_-$ by the assumption made on $\Gscr_2^k$ and $K_0^k$. Furthermore, since $\sigma(\Gscr_1) = \{i\omega_k\}_{k=1}^q$, the operator $\Gscr_1 + \epsilon \Gscr_2\Cscr\tilde H$ satisfies the stability conditions of the operator $A_c - \epsilon\tilde PK$ in \cite[Appendix B]{HamPoh11}. Hence, by \cite[Appendix B]{HamPoh11} there exist constants $M_1, \beta > 0$ such that for all $\epsilon > 0$ sufficiently small we have $\|\mathbb{T}_{2}(t)\| \leq M_1 e^{-\epsilon\beta t}$ for $t \geq 0$, where $\mathbb{T}_2$ is the semigroup generated by $\Gscr_1 + \epsilon \Gscr_2\Cscr\tilde H$. This further implies that
$$
	\|(\lambda - \Gscr_1 + \epsilon\Gscr_2\Cscr\tilde H)^{-1}\| \leq \frac{M_1}{\epsilon\beta},\qquad
		\lambda \in \overline\C_+. 
$$

Consider the operator $\hat A_e$ in the form $A_1 + \epsilon^2A_2$. Since we have shown that the diagonal operators of $A_1$ generate exponentially stable semigroups and since $\Cscr$ is admissible for $A_s$, it follows that $A_1$ is the generator of an exponentially stable semigroup. Furthermore, there exists an $M_2 > 0$ such that for all $\epsilon > 0$ sufficiently small, the estimate $\|(\lambda - A_1)^{-1}\|\leq M_2/\epsilon$ holds for all $\lambda \in \overline \C_+$. Since $A_2$ is bounded, this implies that
$$
	\|\epsilon^2A_2(\lambda - A_1)^{-1}\| \leq \epsilon \|A_2\|M_2,\qquad
		\lambda\in\overline\C_+,
$$
so that for $\epsilon < (\|A_2\|M_2)^{-1}$ we have $\|\epsilon^2A_2(\lambda - A_1)^{-1}\| < 1$ on the closed right half plane. Using another Neumann series expansion, we obtain that
$$
	(\lambda - \hat{A}_e)^{-1}= 
	(\lambda - A_1)^{-1}(1 - \epsilon^2A_2(\lambda - A_1)^{-1})^{-1}
$$
is uniformly bounded on $\overline\C_+$.

Thus, by the preceding argument  there exists an $\epsilon^* > 0$ such that the resolvent of $\hat A_e$ is uniformly bounded on $\overline\C_+$ for all $0<\epsilon<\epsilon^*$. By the Gearhart-Pr\"uss-Greiner theorem, the semigroup $\hat{\mathbb{T}}_e$ generated by $\hat{A}_e$ is exponentially stable, and therefore, the semigroup $R\hat{\mathbb{T}}_eR^{-1}$ generated by $A_e$ is exponentially stable as well, for all $0<\epsilon<\epsilon^*$.
\end{proof}
\label{lem:clstab}
\end{lemma}

\subsection{A regulating controller} \label{ROR:nrrc}

The following theorem gives necessary and sufficient conditions for a controller to achieve output regulation for the plant \eqref{eq:plant}, i.e., a criterion equivalent to the solvability of the regulator equations. The result extends \cite[Thm 5.1]{PauPoh14} to boundary control systems. 

\begin{theorem}
Assume that the closed-loop system is regular and exponentially stabilized by the controller $(\Gscr_1, \Gscr_2, K, Q)$. Then the controller solves the output regulation problem if and only if the equations
\begin{subequations}
\begin{align}
	P_s(i\omega_k) K z_k & = -P_0(i\omega_k)E_s\phi_k - F\phi_k \label{eq:nrc1} \\
	(i\omega_k - \Gscr_1)z_k & = 0 \label{eq:nrc2}
\end{align}
\end{subequations}
have solutions $z_k \in Z$ for all $k \in \{1,2,\ldots, q\}$, where $\{\phi_k\}_{k=1}^q$ is the Euclidean basis of $\mathbb{C}^q$. Furthermore, the solutions $z_k$ are unique when they exist.
\begin{proof}
Let us first assume that the controller solves the output regulation problem, i.e., by Theorem \ref{thm:PauPoh14_4.1} the  regulator equations have a solution $\Sigma = (\Pi, \Gamma)^T \in \Lscr(W, X_e)$. Let $k \in \{1,2,\ldots, q\}$ be arbitrary. As $\phi_k$ is an eigenvector of $S$, applying the Sylvester equation $\Sigma S = A_e\Sigma + B_e$ to $\phi_k$ yields $(i\omega_k -  A_e)\Sigma \phi_k =  B_e \phi_k$, i.e.,
\begin{align*}
& \bbm{(i\omega_k - A_s + B_sR_1K\Gscr_2\Cscr)\Pi\phi_k - (\Ascr B_sR_1K - B_sR_1K\tilde\Gscr_1)\Gamma \phi_k \\ -\Gscr_2\Cscr\Pi\phi_k + (i\omega_k - \tilde\Gscr_1)\Gamma\phi_k} \\ 
& = \bbm{(\Ascr B_sE_s - B_sE_sS - B_sR_1K\Gscr_2(\Cscr B_sE_s + F))\phi_k \\ \Gscr_2(\Cscr B_sE_s + F)\phi_k}.
\end{align*}
where we again denote $\tilde\Gscr_1 := \Gscr_1+\Gscr_2\Cscr B_sR_1K$. The second line implies 
\begin{equation}
(i\omega_k - \Gscr_1)\Gamma\phi_k = \Gscr_2(\Cscr\Pi + \Cscr B_sR_1 K \Gamma + (\Cscr B_sE_s + F))\phi_k.
\label{eq:req1line2}
\end{equation}
Now, as applying the second regulator equation to $\phi_k$ yields 
\begin{equation}
\label{eq:reg2}
0 = C_e\Sigma\phi + D_e\phi_k = \Cscr\Pi\phi_k + \Cscr B_sR_1 K\Gamma\phi_k + (\Cscr B_s E_s + F)\phi_k,
\end{equation}
it follows from \eqref{eq:reg2} and \eqref{eq:req1line2} that $(i\omega_k - \Gscr_1)\Gamma \phi_k = 0$. If we choose $z_k = \Gamma \phi_k$, then \eqref{eq:nrc2} follows immediately. Furthermore, from \eqref{eq:reg2} we obtain 
\begin{equation}
\label{eq:foo2}
\Cscr \Pi \phi_k = -\Cscr B_sR_1 K \Gamma \phi_k - (\Cscr B_sE_s + F)\phi_k.
\end{equation} 
Substituting $\Cscr\Pi\phi_k$ for \eqref{eq:foo2} in the first line of the Sylvester equation yields
\begin{equation}
(i\omega_k - A_s)\Pi \phi_k - \Ascr B_s R_1K\Gamma \phi_k + B_sR_1K\Gscr_1\Gamma \phi_k = (\Ascr B_sE_s - B_sE_sS)\phi_k,
\label{eq:reg1line1mod}
\end{equation}
and utilizing $S\phi_k = i\omega_k\phi_k$ and $\Gscr_1\Gamma \phi_k = i\omega\Gamma\phi_k$, we obtain from \eqref{eq:reg1line1mod} that
\begin{equation}
\Pi \phi_k = (i\omega_k - A_s)^{-1}(\Ascr B_s - i\omega_k B_s)(R_1K\Gamma \phi_k + E_s\phi_k).
\label{eq:foo}
\end{equation}
Finally, substituting $\Pi\phi_k$ for \eqref{eq:foo} in \eqref{eq:reg2} yields
$$
 0 = P_s(i\omega_k)K \Gamma \phi_k + P_0(i\omega_k)E_s\phi_k + F\phi_k,
$$
from which \eqref{eq:nrc1} follows as we chose $z_k = \Gamma\phi_k$.

Now assume that equations \eqref{eq:nrc1}--\eqref{eq:nrc2} have solutions $z_k \in Z$. Define $\Pi \in \Lscr(W, X), \Gamma \in \Lscr(W, Z)$ and $\Sigma = (\Pi, \Gamma)^T$ by
\begin{equation}
\begin{split}
	\Gamma & := \sum_{k=1}^q \langle \cdot, \phi_k \rangle z_k, \\
	 \Pi & := \sum_{k=1}^q \langle \cdot, \phi_k \rangle (i\omega_k - A_s)^{-1}(\Ascr B_s - i\omega_k B_s)(R_1Kz_k + E_s\phi_k).
	 \end{split}
\label{eq:gammapi}
\end{equation}
The definitions imply that $\Rscr(\Sigma) \subset \Dscr(  A_e ) \subset \Dscr (  C_e)$, and we will show that $\Sigma$ is the solution of the regulator equations.

Let $k \in \{1,2,\ldots, q\}$ be arbitrary. Considering the first line of $(i\omega_k -  A_e)\Sigma \phi_k -  B_e\phi_k$, we obtain using \eqref{eq:nrc2}, $S\phi_k = i\omega_k$, the definition of $\Pi$, and  \eqref{eq:nrc1} that
$$
\begin{aligned}
 &\quad  (i\omega_k - A_s)\Pi\phi_k + B_sR_1 K\Gscr_2 \Cscr \Pi \phi_k \\
 & \quad - (\Ascr B_s R_1K - B_sR_1K(\Gscr_1 + \Gscr_2\Cscr B_sR_1K))\Gamma\phi_k \\
& \quad- (\Ascr B_s E_s - B_sE_sS - B_sR_1K\Gscr_2(\Cscr B_sE_s + F))\phi_k \\
& =  B_sR_1K\Gscr_2(\Cscr \Pi\phi_k + \Cscr B_sR_1 K\Gamma\phi_k + \Cscr B_sE_s\phi_k + F\phi_k) \\
& =  B_sR_1K\Gscr_2(P_s(i\omega_k)K\Gamma\phi_k + P_0(i\omega_k)E_s\phi_k + F\phi_k) = 0.
\end{aligned}
$$
Note that by \eqref{eq:nrc1} we also have 
\begin{align*}
	C_e\Sigma\phi_k + D_e\phi_k & = \Cscr \Pi\phi_k + \Cscr B_s R_1K\Gamma\phi_k + \Cscr B_sE_s\phi_k + F\phi_k \\
	& = P_s(i\omega_k)K\Gamma\phi_k + P_0(i\omega_k)E_s\phi_k + F\phi_k = 0,
\end{align*}
i.e., $\Sigma$ solves the second regulator equation. Finally, the second line of $(i\omega_k -  A_e)\Sigma \phi_k -  B_e\phi_k$ yields
$$
\begin{aligned}
	& -\Gscr_2\Cscr \Pi\phi_k + (i\omega_k - \Gscr_1)\Gamma\phi_k - \Gscr_2\Cscr B_sR_1K\Gamma\phi_k - \Gscr_2(\Cscr B_s E_s + F)\phi_k \\
	= & -\Gscr_2(\Cscr\Pi\phi_k + \Cscr B_sR_1K\Gamma\phi_k + \Cscr B_sE_s\phi_k + F\phi_k) = 0.
	\end{aligned}
$$
Thus, as $\{\phi_k\}_{k=1}^q$ is a basis of $\mathbb{C}^q$ and the choice of $k$ was arbitrary, $\Sigma$ is the solution of the regulator equations $\Sigma S =  A_e \Sigma +  B_e$ and $ C_e \Sigma +  D_e = 0$. Now, by Theorem \ref{thm:PauPoh14_4.1}, the controller solves the output regulation problem.

It yet remains to prove the uniqueness of the solutions $z_k$ of \eqref{eq:nrc1}--\eqref{eq:nrc2}. Let $z_k$ and $z_k'$ be two solutions of \eqref{eq:nrc1}--\eqref{eq:nrc2}, and use \eqref{eq:gammapi} to define $\Sigma = (\Pi, \Gamma)^T$ and $\Sigma' = (\Pi', \Gamma')^T$ corresponding to $z_k$ and $z_k'$, respectively. It now follows from the above proof that both $\Sigma$ and $\Sigma'$ satisfy the Sylvester equation, and by the uniqueness of the solution of the Sylvester equation we must have $\Sigma = \Sigma'$. In particular, $z_k = \Gamma\phi_k = \Gamma'\phi_k = z_k'$, i.e., the solutions $z_k$ of \eqref{eq:nrc1}--\eqref{eq:nrc2} are unique.
\end{proof}
\label{thm:nonrobustcontrol}
\end{theorem}

Based on Theorem \ref{thm:nonrobustcontrol}, we can now construct a regulating controller for the plant \eqref{eq:plant}. Choose $Z = W$ and choose the controller parameter $Q \in \Lscr(Y,U)$ such that the semigroup $\mathbb{T}_s$ generated by $A_s$ is exponentially stable and $\Cscr$ is an admissible observation operator for $A_s$. Choose the remaining parameters as
\begin{subequations}
\begin{align}
	\Gscr_1 & = S = \operatorname{diag}(i\omega_1,i\omega_2,\ldots, i\omega_q),  \label{eq:nrrcG1} \\
	K & = \epsilon K_0 = \epsilon\left[u_1, u_2, \ldots, u_q\right],  \label{eq:nrrcK} \\
	\Gscr_2 & = (\Gscr_2^k)_{k=1}^q = (-(P_s(i\omega_k)u_k)^*)_{k=1}^q, \label{eq:nrrcG2}
\end{align}
\end{subequations}
where $\epsilon > 0$ is called the tuning parameter and $u_k \in U$ are chosen such that \cite[Sec. 4.2]{Pau17b}
\begin{equation}
	\begin{cases}
		P_s(i\omega_k)u_k = y_k, & y_k \neq 0, \\
		u_k \notin \Nscr(P_s(i\omega_k))\text{  arbitrary}, & y_k = 0,
	\end{cases}
	\label{eq:uk}
\end{equation}
where we denote $y_k  = -P_0(i\omega_k)E_s\phi_k - F\phi_k$. For this to be possible, we need to assume that $P_s(i\omega_k) \neq 0$ and $y_k \in \Rscr(P_s(i\omega_k))$ for all $k \in \{1,2,\ldots,q\}$, so that  there exist some $u_k\in U$ satisfying \eqref{eq:uk}. However, this assumption is also necessary for the solvability of the output regulation problem by Theorem \ref{thm:nonrobustcontrol}.

\begin{proposition}
There exists an $\epsilon^* > 0$ such that the controller with the parameter choices \eqref{eq:nrrcG1}--\eqref{eq:nrrcG2} solves the output regulation problem for all $0 < \epsilon < \epsilon^*$.
\label{thm:nrrc}
\begin{proof}
First of all we note that the choices of $\Gscr_1$ and $K$ imply that the equations \eqref{eq:nrc1}--\eqref{eq:nrc2} have the solutions $z_k = \epsilon^{-1}\phi_k$ if $P_0(i\omega_k)E_s\phi_k + F\phi_k \neq 0$ or $z_k = 0$ otherwise. Now, as $Q$ exponentially stabilizes the plant and $\Cscr$ is admissible for $A_s$, $\sigma(\Gscr_1) = \{i\omega_k\}_{k=1}^q$, and
$$
	\sigma(\Gscr_2^kP_s(i\omega_k)K_0^k) = \sigma(-(P_s(i\omega_k)u_k)^*P_s(i\omega_k)u_k) \subset \mathbb{C}_-
$$
as $P_s(i\omega_k)u_k \neq 0$ for $k \in \{1,2,\ldots,q\}$,  we have by Lemma \ref{lem:clstab} that there exists an $\epsilon^* > 0$ such that the closed-loop system is exponentially stable for all $0 < \epsilon < \epsilon^*$. Thus, by Theorem \ref{thm:nonrobustcontrol} the controller solves the output regulation problem.
\end{proof}
\end{proposition}

\subsection{The Internal Model Principle} \label{ROR:ROR}

Before presenting an approximate robust controller in \S\ref{ROR:arrc} and a robust controller in \S\ref{ROR:rrc}, we will present a general result that characterizes robust controllers. That is, we will show that in order for a controller to achieve robust output regulation, it has to contain an  internal model of the dynamics of the exosystem. We will express this using the following $\Gscr$-conditions \cite[Def. 10]{HamPoh10}.
\begin{definition}
	A quadruple of bounded operators $(\Gscr_1, \Gscr_2, K, Q)$ is said to satisfy the $\Gscr$-conditions if
	\begin{subequations}
	\begin{align}
		\Rscr(i\omega_k - \Gscr_1) \cap \Rscr(\Gscr_2) & = \{0\}, \quad \forall k \in \{1,2,\ldots, q\} \label{eq:gcond1} \\
		\Nscr(\Gscr_2) & = \{0\}. \label{eq:gcond2}
	\end{align}
\label{eq:gconditions}%
\end{subequations}
	\label{def:gconditions}
\end{definition}

Note that while the parameters $K$ and $Q$ are not present in the $\Gscr$-conditions, they contribute to exponentially stabilizing the closed-loop system. The sufficiency part of the following result has been presented in the case $R_1=R_2 = I$ in \cite[Thm 4]{HumPau17} and the necessity part extends the results of \cite[Thm 5.2]{PauPoh10} and \cite[Thm 7]{Pau16a} to boundary control systems.

\begin{theorem}
Assume that the closed-loop system is regular and exponentially stabilized by the controller $(\Gscr_1, \Gscr_2, K, Q)$. Then the controller solves the robust output regulation problem if and only if it satisfies the $\Gscr$-conditions.
\begin{proof}
Let us assume that the controller solves the robust output regulation problem and show that \eqref{eq:gconditions} hold starting with \eqref{eq:gcond1}. Let $k \in \{1,2,\ldots,q\}$ be arbitrary and $w \in \Rscr(i\omega_k - \Gscr_1) \cap \Rscr (\Gscr_2)$.  Then there exist $z \in Z$ and $y \in Y$ such that $w = (i\omega_k - \Gscr_1)z = \Gscr_2 y$. Let us leave the operators $(\Ascr, \Bscr, \Cscr)$ unperturbed and choose such perturbations from $\Oscr$ that $E_s' = 0$ and $F' = \langle \cdot, \phi_k \rangle (y - P_s(i\omega_k)Kz)$. Choose $\Sigma = (\Gamma, \Pi)^T \in \Lscr(W,X_e)$ such that
$$
	\Gamma = \langle \cdot, \phi_k \rangle z, \quad \Pi = \langle \cdot, \phi_k\rangle(i\omega_k - A_s)(\Ascr B_s - i\omega_kB_s)R_1Kz,
$$
which can be shown to be the solution of the Sylvester equation by a direct computation. As $C_e\Sigma\phi_k + D_e'\phi_k  = 0$ by the controller solving the robust output regulation problem, we obtain
$$
\begin{aligned}
w & = (i\omega_k - \Gscr_1)z = \Gscr_2 y = \Gscr_2 (P_s(i\omega_k)Kz + F'\phi_k) \\
& = \Gscr_2(\Cscr\Pi\phi_k + \Cscr B_sR_1K\Gamma\phi_k + F'\phi_k) \\
& = \Gscr_2(C_e\Sigma\phi_k + D_e'\phi_k) = 0,
\end{aligned}
$$
and thus $w = 0$, which concludes the first part of the necessity proof.

Let us now show that \eqref{eq:gcond2} holds. Let $y \in \Nscr(\Gscr_2)$ and let $\phi \in W$ be such that $\|\phi\| = 1$. Leave the operators $(\Ascr,\Bscr, \Cscr)$ unperturbed and choose $E' = 0$ and $F' = \langle \cdot, \phi \rangle y \in \Lscr(W,Y)$. If we choose $\Sigma = 0 \in \Lscr(W, X_e)$, for all $v \in W$ we have $\Sigma Sv = 0$ and
$$
	 A_e \Sigma v +  B_e' v = \bbm{- B_sR_1 K\Gscr_2 F'v \\ \Gscr_2 F'v}
	 = \bbm{-\langle v, \phi \rangle B_sR_1 K \Gscr_2 y \\ \langle v, \phi \rangle \Gscr_2 y} = 0,
$$
and thus, $\Sigma = 0$ is the unique solution of the Sylvester equation. As the controller solves the robust output regulation problem, we have by Theorem \ref{thm:PauPoh14_4.1} that 
$$
	0 =  C_e \Sigma \phi +  D_e' \phi = F'\phi = \langle \phi, \phi \rangle y = y,
$$
which concludes the necessity proof. The sufficiency part follows by simple modifications from \cite[Thm. 4]{HumPau17}.
\end{proof}
\label{thm:gcond+expstab->ror}
\end{theorem}

\begin{remark}
Theorem \ref{thm:gcond+expstab->ror} states that any controller that stabilizes a regular closed-loop system exponentially and satisfies the $\Gscr$-conditions solves the robust output regulation problem. In particular, this implies that if a robust regulating controller $(\Gscr_1,\Gscr_2, K, Q)$ is constructed, then every controller $(\Gscr_1,\Gscr_2', K', Q')$, where $(\Gscr_2', K', Q')$ are boundedly perturbed $(\Gscr_2, K, Q)$, solves the robust output regulation problem, provided that the closed-loop system remains exponentially stable and ($\Gscr_1, \Gscr_2')$ satisfy the $\Gscr$-conditions. Note that only rather specific perturbations would be allowed in $\Gscr_1$ as it has to include an exact internal model of the dynamics of the exosystem.
\label{rem:contr_per}
\end{remark}

Note that the rank-nullity theorem and the second $\Gscr$-condition imply that $\dim Z \geq \dim \Rscr(\Gscr_2) = \dim Y$. Thus, if the output space of the system is infinite-dimensional as, e.g., in the wave equation of \S\ref{sec:firstord}, Theorem \ref{thm:gcond+expstab->ror} implies that robust controllers for such systems are necessarily infinite-dimensional. However, we can construct a finite-dimensional controller that solves the robust output regulation problem \emph{approximately}. We will construct such a controller in the next section. Finally, in \S\ref{ROR:rrc} we will construct an infinite-dimensional controller that achieves exact robust output regulation. The following assumption is required for the remaining sections:
\begin{assumption}
  The transfer function $P_s(\lambda)$ is surjective at all the eigenvalues $\{i\omega_k\}_{k=1}^q$ of $S$.
\end{assumption}

\subsection{An approximate robust controller} \label{ROR:arrc}

In this section, we consider approximate robust output regulation on $Y$. We will solve the control problem by choosing a subspace $Y_N$ of $Y$ and constructing a controller that robustly tracks the reference signal projected onto $Y_N$. If $Y_N$ is chosen to be finite-dimensional, we can construct a finite-dimensional robust regulating controller even if the output space of the system is infinite-dimensional. Furthermore, we derive an upper bound for the asymptotic regulation error. Our result generalizes the controller structure presented in \cite[Thm. 3.5]{Pau16barxiv} where discrete-time systems with constant reference signals were considered.

Let $Y_N$ be a closed subspace of $Y$ and choose $Z := Y_N^q$. Choose the controller parameter $Q \in \Lscr(Y,U)$ such that the semigroup $\mathbb{T}_s$ generated by $A_s$ is exponentially stable and $\Cscr$ is an admissible observation operator for $\mathbb{T}_s$. Choose the remaining parameters as
\begin{subequations}
\begin{align}
	\Gscr_1 & = \operatorname{diag}(i\omega_1I_{Y_N},i\omega_2I_{Y_N},\ldots, i\omega_qI_{Y_N}),  \label{eq:acG1} \\
	K & = \epsilon K_0 = \epsilon[K_0^1, K_0^2, \ldots, K_0^q ] \in \Lscr (Z, U),  \label{eq:acK} \\
	\Gscr_2 & = (\Gscr_{20}^kP_N)_{k=1}^q \in \Lscr(Y,Z), \label{eq:acG2}
\end{align}
\end{subequations}
where $P_N$ is a projection onto $Y_N$, and $\Gscr_{20}^k$ and $K_0^k$ are such that 
\begin{equation}
\sigma(\Gscr_{20}^kP_NP_s(i\omega_k)K_0^k) \subset \mathbb{C}_-
\label{eq:approxsc}
\end{equation} 
for all $k \in \{1,2,\ldots,q\}$. We can choose, e.g., $\Gscr_{20}^k = -I_{Y_N}$ and $K_0^k = \left(P_N P_s(i\omega_k)\right)^{[-1]}$ for $k \in \{1,2,\ldots,q\}$, and conversely, the spectrum condition implies that $\Gscr_{20}^k$ and $P_NP_s(i\omega_k)K_0^k$ are boundedly invertible. 

In the following theorem, we show that a controller with the aforementioned structure solves the approximate robust output regulation problem. Furthermore, we will show that for some constants $M,\alpha >0$ and all $t\geq 0$ the regulation error satisfies
\begin{equation}
\label{eq:appr_error}
\int_t^{t+1}\|e(s)\|^2 \ud s \leq Me^{-\alpha t}(\|x_{e0}\|^2 + \|v_0\|^2) + \delta\|v_0\|^2,
\end{equation}
where $x_{e0}$ and $v_0$ are the initial states of the closed-loop system and the exosystem, respectively, and $\delta$ is given by
\begin{equation}
\label{eq:delta}
\delta = \left\|(I - P_N) \sum_{k=1}^q \big(P_s(i\omega_k)Kz_k + P_0(i\omega_k)E_sv_k + Fv_k\big)\right\|^2,
\end{equation}
where $v_k$ are the components of the unit vector $v_{\max} \in W$ satisfying  $\|C_e\Sigma + D_e\| = \|C_e\Sigma v_{\max} + D_e v_{\max}\|_Y$ and $z_k = \Gamma v_k$ where $\Gamma$ is given in \eqref{eq:sylsol_appr}. Note that since $W$ is finite dimensional, $v_{\max}$ is well-defined. Further note that we cannot guarantee pointwise convergence for the regulation error, and therefore the upper bound is presented in the integral form. Finally, since $\sum_{k=1}^q \big(P_s(i\omega_k)Kz_k + P_0(i\omega_k)E_sv_k + Fv_k\big) \in Y$, the projection $P_N$ (or rather the space $Y_N$) can be chosen such that $\delta$ becomes arbitrarily small. We will demonstrate this procedure in \S\ref{sec:waveex} for the wave equation.

\begin{theorem}
\label{thm:approxcontrol}
There exists an $\epsilon^* > 0$ such that for all $0 < \epsilon < \epsilon^*$ the controller with the parameter choices \eqref{eq:acG1}--\eqref{eq:acG2} solves the approximate robust output regulation problem and there exist some constants $M, \alpha > 0$ such that for all $t \geq 0$ the regulation error satisfies \eqref{eq:appr_error}. 

Furthermore, the controller is robust with respect to those perturbations of class $\Oscr$ that give rise to an exponentially stable perturbed closed-loop system, and the regulation error behaves as in \eqref{eq:appr_error} for the perturbed parameters of the plant and the exosystem.

\begin{proof}
By Lemma \ref{lem:clstab}, the closed-loop system is exponentially stable for all sufficiently small $\epsilon > 0$. Thus, as $\sigma(S) \subset i\mathbb{R}$, the Sylvester equation has a unique solution $\Sigma = (\Pi, \Gamma)^T$, and a direct computation using \eqref{eq:transf} verifies that
\begin{equation}
\begin{split}
	\Gamma & = (\Gamma_k)_{k=1}^q = -\epsilon^{-1}\big(\langle \cdot, \phi_k \rangle (P_N P_s(i\omega_k)K_0^k)^{-1}P_N(P_0(i\omega_k)E_s + F)\phi_k\big)_{k=1}^q, \\
	\Pi & = \sum_{k=1}^q \langle \cdot, \phi_k\rangle(i\omega_k - A_s)^{-1}(\Ascr B_s - i\omega_k B_s)(R_1K\Gamma +  E_s)\phi_k
\end{split}
\label{eq:sylsol_appr}%
\end{equation}
solves $\Sigma S\phi_k = A_e\Sigma \phi_k + B_e \phi_k$, i.e., $(i\omega_k - A_e)\Sigma \phi_k = B_e \phi_k$ for all $k \in \{1,2,\ldots,q\}$. Here one also uses that our $\Gamma$ satisfies
\begin{equation}
	P_NP_s(i\omega_k)K\Gamma\phi_k  = \epsilon P_NP_s(i\omega_k)K_0^k\Gamma_k\phi_k =-P_NP_0(i\omega_k)E_s\phi_k - P_NF\phi_k.
\label{eq:reg2appr}
\end{equation}
Note that \eqref{eq:sylsol_appr} is well-defined and bounded since $P_NP_s(i\omega_k)K_0^k$ are boundedly invertible by \eqref{eq:approxsc}.

Let us now consider the behavior of the regulation error. By \cite[Lem. 4.3]{PauPoh14}, we may write 
$$
	e(t) = C_e\mathbb{T}_e(t)(x_{e0} - \Sigma v_0) + (C_e\Sigma + D_e)v(t),
$$
and we obtain that for all $t \geq 0$
$$
\begin{aligned}
   \int_t^{t+1}\|e(s)\|^2\ud s
	& = \int_t^{t+1} \|C_e\mathbb{T}_e(s)(x_{e0} - \Sigma v_0) + (C_e\Sigma + D_e)v(s)\|^2\ud s \\
	& \leq Me^{-\alpha t}(\|x_{e0}\|^2 + \|v_0\|^2) + \|C_e\Sigma + D_e\|^2\|v_0\|^2
\end{aligned}
$$
for some $M,\alpha > 0$ as $\Sigma$ is bounded, $\mathbb{T}_e$ is exponentially stable, $C_e$ is admissible for $\mathbb{T}_e$, and due to the structure of the signal generator $\|v(t)\| = \|e^{St}v_0\| = \|v_0\|$.

We will show that
$$
C_e\Sigma v_{\max} + D_ev_{\max} = (I - P_N)\sum_{k=1}^q \big(P_s(i\omega_k)Kz_k + P_0(i\omega_k)E_sv_k + Fv_k\big).
$$
A direct computation using \eqref{eq:sylsol_appr} shows that
\begin{equation}
C_e\Sigma v_{\max} + D_e v_{\max} = \sum_{k=1}^q \big(P_s(i\omega_k)K\Gamma v_k + P_0(i\omega_k)E_sv_k + Fv_k\big),
	\label{eq:appr_err1}
\end{equation}
Denoting $z_k = \Gamma v_k$, we have by \eqref{eq:reg2appr} that
\begin{equation}
	P_NP_s(i\omega_k)Kz_k = -P_N P_0(i\omega_k)E_sv_k - P_N Fv_k,
\label{eq:actmp2}
\end{equation}
and now, combining \eqref{eq:actmp2} with \eqref{eq:appr_err1} yields
$$
	C_e \Sigma v_{\max} +  D_ev_{\max} = (I - P_N)\sum_{k=1}^q \big(P_s(i\omega_k)Kz_k + P_0(i\omega_k)E_sv_k + Fv_k\big),
$$
which implies \eqref{eq:delta}, and thus, \eqref{eq:appr_error}.

If the parameters $(\Ascr, \Bscr, \Cscr, E, F)$ are perturbed in such a way that the closed-loop system remains exponentially stable, then the regulation error asymptotically satisfies $\int_t^{t+1}\|e(s)\|^2\ud s \leq M'e^{-\alpha't}(\|x_{e0}\|^2 + \|v_0\|^2) + \|C_e'\Sigma' + D_e'\|^2\|v_0\|^2$ for all $t \geq 0$, where $M',\alpha' > 0$, and $C_e', D_e'$ and $\Sigma'$ are related to the the perturbed closed-loop system. By repeating the above computations with the perturbed parameters we clearly obtain $C_e'\Sigma'v_{\max}' + D_e'v_{\max}' = (I-P_N)\sum_{k=1}^q P_s'(i\omega_k)Kz_k' + P_0'(i\omega_k)E_s'v_k' + F'v_k')$, where $z_k'$ is the unique solution of $P_NP_s'(i\omega_k)Kz_k' = -P_NP_0'(i\omega_k)E_s'v_k' - P_NF'v_k'$ in $\Nscr(i\omega_k - \Gscr_1)$. Thus, the controller approximately solves the robust output regulation problem.
\end{proof}
\end{theorem}

\begin{remark}
As an alternative to the error estimate given in \eqref{eq:appr_error}, one can make a coarser choice for $\delta$ that does not require $v_{\max}$:
$$
	\delta = \sum_{k=1}^q \left|\left|(I - P_N)\big(P_s(i\omega_k)Kz_k + P_0(i\omega_k)E_s\phi_k + F\phi_k\big)\right|\right|^2,
$$
where $\{\phi_k\}_{k=1}^q$ is the Euclidean basis of $W$ and $z_k = \Gamma\phi_k$.
\end{remark}

\begin{corollary} \label{cor:approx}
In Theorem \ref{thm:approxcontrol}, the regulation error satisfies $e^{\beta\cdot}P_Ne(\cdot) \in L^2([0, \infty); Y)$ for some $\beta > 0$ independent of $x_{e0} \in X_e$ and $v_0 \in W$. Under perturbations of class $\Oscr$ that give rise to an exponentially stable closed-loop system, the regulation error satisfies $e^{\beta'\cdot}P_Ne(\cdot) \in L^2([0, \infty); Y)$ for some $\beta' > 0$ independent of $x_{e0} \in X_e$ and $v_0 \in W$.

\begin{proof}
Let us first show that $P_NC_e\Sigma + P_N D_e = 0$. A direct computation using \eqref{eq:sylsol_appr} together with \eqref{eq:reg2appr} shows that for all $k \in \{1,2,\ldots, q\}$:
$$
	P_NC_e\Sigma \phi_k  + P_ND_e\phi_k
	  = P_NP_s(i\omega_k)K\Gamma\phi_k + P_NP_0(i\omega_k)E_s\phi_k + P_NF\phi_k = 0,
$$
and as $\{\phi_k\}_{k=1}^q$ form a basis of $\mathbb{C}^q$, we have that $P_NC_e\Sigma  + P_ND_e = 0$. By the proof of Theorem \ref{thm:approxcontrol} we now have for some $\beta > 0$ that
$$
\begin{aligned}
\int_t^{t+1}\|e^{\beta s}P_Ne(s)\|^2 \ud s & \leq e^{\beta(t+1)}\int_t^{t+1}\|P_Ne(s)\|^2\ud s \\
&  \leq e^\beta Me^{(\beta - \alpha)t}(\|x_{e0}\|^2 + \|v_0\|^2),
\end{aligned}
$$
so for any $0 < \beta < \alpha$ we obtain
$$
\begin{aligned}
	\int_0^\infty \|e^{\beta s}P_Ne(s)\|^2\ud s & \leq  e^\beta M(\|x_{e0}\|^2 + \|v_0\|^2) \sum_{t=0}^\infty e^{(\beta-\alpha) t}\\
	& =  \frac{e^\beta M(\|x_{e0}\|^2 + \|v_0\|^2)}{1 - e^{\beta-\alpha}},
\end{aligned}
$$
by which $e^{\beta \cdot}P_Ne(\cdot) \in L^2([0, \infty))$ for any $0 < \beta < \alpha$. By the robustness part of Theorem \ref{thm:approxcontrol}, the same holds for some $0 < \beta'< \alpha'$ under perturbations of class $\Oscr$ that give rise to an exponentially stable closed-loop system.
\end{proof}
\end{corollary}

\subsection{A robust controller} \label{ROR:rrc}

In this section, we utilize the approximate controller structure of the previous section to construct an exact robust controller which, however, necessarily has infinite-dimensional state space if the output space of the plant is infinite-dimensional. Thus, we choose $Z = Y^q$ and choose the controller parameter $Q \in \Lscr(U,Y)$ such that the semigroup $\mathbb{T}_s$ generated by $A_s$ is exponentially stable and $\Cscr$ is an admissible observation operator for $\mathbb{T}_s$. Following \cite[Sec. IV]{Pau16a} or \cite[Thm. 8]{HumPau17}, we choose the remaining parameters as
\begin{subequations}
\begin{align}
	\Gscr_1 & = \operatorname{diag}\left(i\omega_1I_Y,i\omega_2I_Y, \ldots, i\omega_qI_Y\right) \in \Lscr(Z),  \label{eq:rrcG1} \\
	K & = \epsilon K_0 = \epsilon\left[K_0^1,K_0^2, \ldots, K_0^q\right] \in \Lscr(Z,U),  \label{eq:rrcK} \\
	\Gscr_2 & = (-(P_s(i\omega_k)K_0^k)^*)_{k=1}^q 
	 \in \Lscr(Y,Z). \label{eq:rrcG2}
\end{align}
\end{subequations}
Above the components $K_0^k$ can be chosen freely provided that $P_s(i\omega_k)K_0^k$ are invertible. If we choose $K_0^k = P_s(i\omega_k)^{[-1]}$, then $\Gscr_2^k = -I_Y$ for all $k \in \{1,2,\ldots q\}$, then the controller is the same as the approximate controller for the choice $Y_N = Y$. The following result follows immediately from Corollary \ref{cor:approx}.

\begin{corollary}
There exists an $\epsilon^* > 0$ such that a controller with the parameter choices given in \eqref{eq:rrcG1}--\eqref{eq:rrcG2} solves the robust output regulation problem for all $0 < \epsilon < \epsilon^*$.
\label{cor:robregcontr}
\end{corollary}

\begin{remark}
The above result also follows from Lemma \ref{lem:clstab} and Theorem \ref{thm:gcond+expstab->ror} as  the choice $K_0^k = P_s(i\omega_k)^{[-1]}$ yields $\sigma(\Gscr_2^kP_s(i\omega_k)K_0^k) = \sigma(-I_Y) \subset \mathbb{C}_-$, which together with the choice of $Q$ completes the assumptions of Lemma \ref{lem:clstab}, by which the closed-loop system is exponentially stable. Furthermore, it has been shown in the proof of \cite[Thm 8]{Pau16a} that $\Gscr_1$ and $\Gscr_2$ in \eqref{eq:rrcG1}--\eqref{eq:rrcG2} satisfy the $\Gscr$-conditions, and thus,  the controller solves the robust output regulation problem by Theorem \ref{thm:gcond+expstab->ror}.
\end{remark}

\section{Approximate robust regulation of the wave equation} \label{sec:waveex} 

Consider the wave equation as given in \eqref{eq:waveintro} with the spatial domain $\Omega := \set{\zeta \in \mathbb{R}^2 \mid 1 < \|\zeta\| < 2}$. Choose the partition $\partial \Omega = \Gamma_0 \cup \Gamma_1$ where $\Gamma_0 = \set{\zeta \in \partial\Omega \mid \|\zeta\| = 1}$ and $\Gamma_1 = \set{\zeta \in \partial\Omega \mid \|\zeta\| = 2}$ which satisfies the assumption in \eqref{eq:bdrpart}, e.g., for $\zeta_0 = 0$, and thus the results presented in Section \ref{sec:damper} are applicable.

For the approximate robust output regulation problem, let $\delta = 0.01$ be given. We choose the output space as $Y := L^2(\Gamma_1)$ which is equivalent to $L^2([0, 2\pi])$. Thus, for the finite-dimensional closed subspace $Y_N$ we may choose, e.g., 
$$
	Y_N := \operatorname{span}\set{1,\cos(k\cdot),\sin(k\cdot)\, | \, k = 1,\ldots,N},
$$
and the projection $P_N$ from $Y$ onto $Y_N$ is then given by
\begin{equation}
  \label{eq:PN}
	P_N y := \frac{1}{\sqrt{2\pi}}\left\langle y, 1 \right\rangle + \frac{1}{\sqrt{\pi}} \sum_{k=1}^N \left(\left\langle y, \cos(k\cdot)\right\rangle
	+ \left\langle y, \sin(k\cdot)\right\rangle\right).
\end{equation}
By standard Fourier analysis, it holds that for all $f \in L^2([0, 2\pi])$, we have
$\displaystyle \lim_{N \to \infty} \left\|(1 - P_N)f \right\| = 0$, and thus, by Theorem \ref{thm:approxcontrol}, for a given reference $y_{ref}$, we can choose $N$ in \eqref{eq:PN} sufficiently large such that asymptotically the regulation error becomes smaller than $\delta\|v_0\|^2$ (in the $L^2$-sense).

Let the reference and disturbance signals be given by
\begin{align*}
y_{ref}(\theta, t)&  = -\frac{1}{2\pi^2}(\pi - \theta)^2\sin(\pi t) -
         \frac{1}{2}\sin\left(\frac{\theta}{2}\right)\cos(2\pi t) \\
  d(\theta, t) & = \cos(\theta)\sin(2\pi t) + \sin(\theta)\sin(\pi t)
\end{align*}
and the disturbance $d$ acts on $\Gamma_1$. Thus, we choose $S = \operatorname{diag}(-2i\pi, -i\pi , i\pi, 2i\pi)$, and
the operators $E$ and $F$ are chosen such that $y_{ref}= -Fv$ and $d = Ev$ for $v_0 = 1$. The controller parameter $Q$ is chosen as $Q(\theta) = 3$, and according to \S\ref{ROR:arrc} we choose
\begin{subequations}
\begin{align}
	\Gscr_1 & = \operatorname{diag}(-2i\pi I_{Y_N},-i\pi I_{Y_N},i\pi I_{Y_N}, 2i\pi I_{Y_N}), \\
	K & = \epsilon\left[K_0^1, K_0^2, K_0^3, K_0^4\right] \\
	\Gscr_2 & = (-P_N)_{k=1}^4
\end{align}
\end{subequations}
where $K_0^k = \left(P_NP_s(i\omega_k)\right)^{[-1]}$, $N=5$ and $\epsilon = 0.15$.

For simulation, the operators related to the wave equation are approximated by the orthonormal eigenfunctions of the Laplacian $\Delta$ with homogeneous boundary conditions. In polar coordinates, these are of the form
\begin{align*}
\phi_{n0}^1(r) & = \frac{1}{\sqrt{2\pi}} \varphi_{n0}(r), \quad & n \in\mathbb{N} \\
\phi_{nm}^1(r,\theta) &  = \frac{1}{\sqrt{\pi}}\varphi_{nm}(r)\cos(m\theta), \quad & m,n\in\mathbb{N} \\
\phi_{nm}^2(r,\theta) & = \frac{1}{\sqrt{\pi}}\varphi_{nm}(r)\sin(m\theta), \quad & m,n\in\mathbb{N},
\end{align*}
where $\varphi_{nm}(r)$ are the appropriately normalized Bessel functions corresponding to the radial part of the Laplacian such that the functions $\{\phi_{nm}^{1,2}\}$ form an orthonormal basis of $L^2(\Omega)$. The eigenvalues are computed numerically and in the simulation we use $n=8$ radial and $m+1=12$ angular eigenfunctions corresponding to the eigenvalues. The transfer function $P_s$ is computed using the approximated operators, and the initial conditions are given by $x_0 = 0$ and $z_0 = 0$. 

In Figure \ref{fig:y}, the output profile $y$ of the controlled wave equation and the reference profile $y_{ref}$ are displayed for $t \in [0, 10]$\footnote{Matlab codes for the simulation are available at \url{https://codeocean.com/capsule/93a9a1bb-8511-4fd1-8030-98270c1c4cde/}}. It can be seen that the output starts to follow the reference signal rather soon, even though some undershooting can be observed throughout the simulation.

\begin{figure}[htp]
\centerline{\includegraphics[width=\columnwidth]{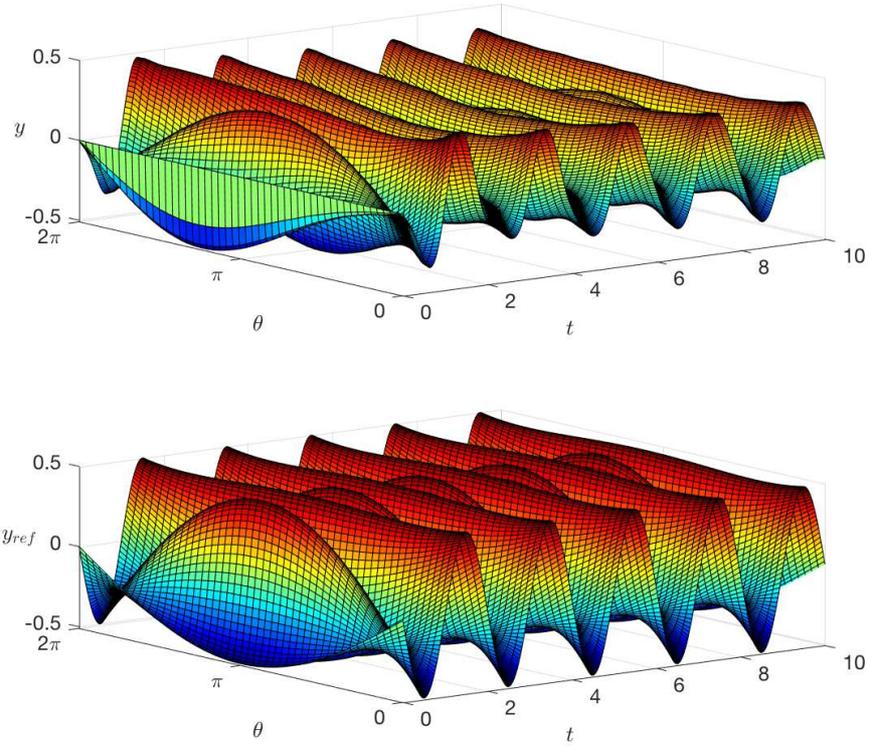}}
\caption[]{The output profile $y$ of the controlled wave equation and the reference profile $y_{ref}$ for $t \in [0, 10]$ and in the same scales.}
  \label{fig:y}
\end{figure}

In Figure \ref{fig:e}, the time average of the norm of the regulation error is displayed for $t \in [0, 20]$. Here it can be seen that, apart from the oscillations and initial errors, the regulation error decays at an exponential rate and that asymptotically it decays beyond the given $\delta\|v_0\|^2$. In Figure \ref{fig:x20}, the wave profile of the controlled system is displayed at time $t = 9$ and in Figure \ref{fig:d}, the disturbance signal is displayed for $t \in [0, 6]$.

\begin{figure}[htp]
\centerline{\includegraphics[width=\columnwidth]{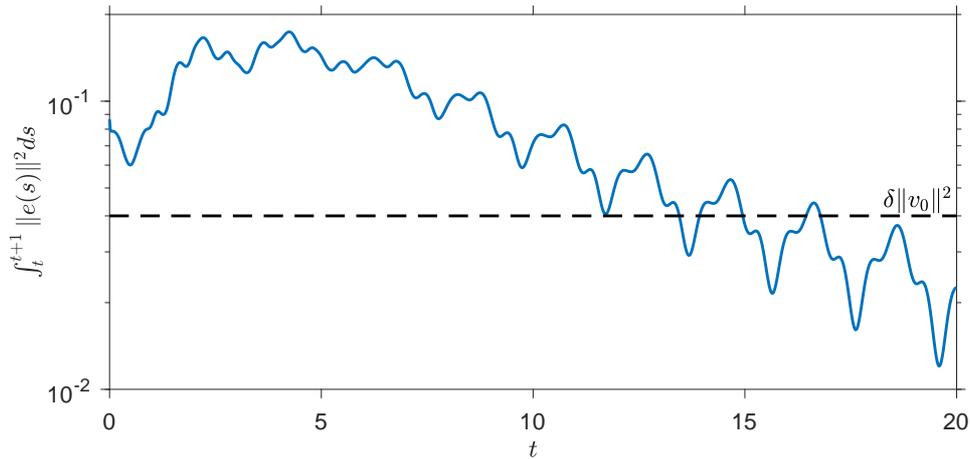}}
\caption[]{The regulation error $\displaystyle \int\limits_t^{t+1}\|y(s) - y_{ref}(s)\|^2ds$ for $t \in [0, 20]$.}
  \label{fig:e}
\end{figure}

\begin{figure}[htp]
\centerline{\includegraphics[width=\columnwidth]{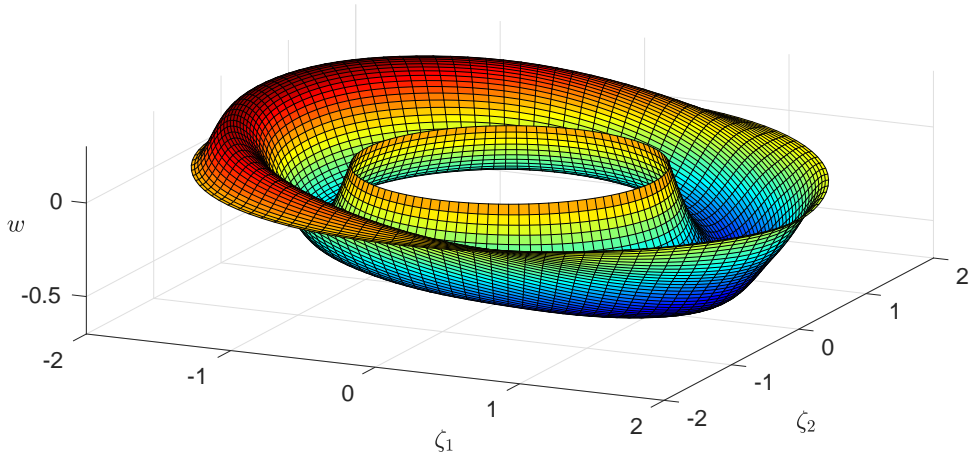}}
\caption[]{The wave profile of the controlled system at $t = 9$.}
  \label{fig:x20}
\end{figure}

\begin{figure}[!htp]
\centerline{\includegraphics[width=\columnwidth]{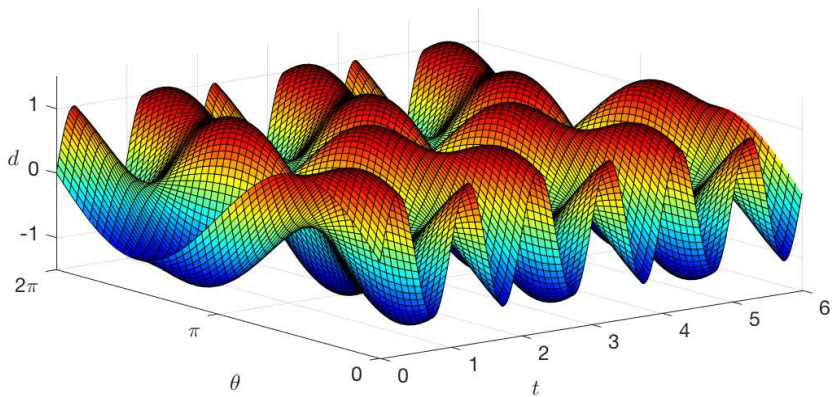}}
\caption[]{The disturbance signal $d$ for $t \in [0, 6]$.}
  \label{fig:d}
\end{figure}

\section{Conclusions} \label{sec:concl}

We developed output regulation for abstract boundary control systems, parametrizing all regulating and robust regulating controllers, and also suggesting some particular choices of such controllers. Since the internal model principle implies that the state space of any robust controller for a system with infinite-dimensional output space has infinite dimension, we extended the concept of approximate robust output regulation to boundary control systems. We demonstrated that approximate robust regulation can be achieved with a finite-dimensional controller by constructing such a controller for the two-dimensional wave equation and demonstrating its performance with numerical simulations.

\def\cprime{$'$} \def\cprime{$'$}

\end{document}

%% file: MKprencls.tex
\usepackage{amsmath}
\usepackage{amsthm}
\usepackage{amssymb}
\usepackage{srcltx}
\usepackage{hyperref}
\usepackage{tikz}

\newtheorem{lemma}{Lemma}[section]
\newtheorem{theorem}[lemma]{Theorem}
\newtheorem{corollary}[lemma]{Corollary}
\newtheorem{proposition}[lemma]{Proposition}

\theoremstyle{definition}

\newtheorem{definition}[lemma]{Definition}
\newtheorem{remark}[lemma]{Remark}
\newtheorem{assumption}[lemma]{Assumption}

\numberwithin{equation}{section}

\newcommand{\leaveout}[1]{}

\makeatletter
\renewcommand{\p@enumii}{}
\makeatother

\newcommand{\DD}[1]{\mathbin{\frac{\rm d }{{\rm d }#1}}}
\newcommand{\ddt}{\DD t}

\newcommand{\ud}{\,{\mathrm d}}

\newcommand{\Div}{{\rm div\,}}
\newcommand{\Grad}{{\nabla }}

\newcommand\R{{\mathbb R}}
\newcommand\C{{\mathbb C}}

\newcommand\Z{{\mathbb Z}}
\newcommand\rplus{{\R_{+}}}

\newcommand\T{{\mathbb T}}

\newcommand{\Ascr}{\mathcal A}
\newcommand{\Bscr}{\mathcal B}
\newcommand{\Cscr}{\mathcal C}
\newcommand{\Dscr}{\mathcal D}

\newcommand{\Gscr}{\mathcal G}
\newcommand{\Hscr}{\mathcal H}

\newcommand{\Lscr}{\mathcal L}
\newcommand{\Nscr}{\mathcal N}
\newcommand{\Oscr}{\mathcal O}

\newcommand{\Rscr}{\mathcal R}

\newcommand{\Uscr}{\mathcal U}

\newcommand{\Yscr}{\mathcal Y}

\newcommand{\Wscr}{\mathcal W}

\newcommand{\dom}[1]{\mathrm{dom}\left(#1\right)}

\newcommand{\Ker}[1]{\ker\left(#1\right)}

\newcommand{\res}[1]{\rho\left(#1\right)}
\newcommand{\re}[0]{\mathrm{Re}\,}

\newcommand{\Ipdp}[2]{\left\langle #1 , #2 \right\rangle}

\newcommand{\set}[1]{\left\lbrace #1 \right\rbrace}

\newcommand{\bigmid}{\bigm\vert}
\newcommand{\Bigmid}{\Bigm\vert}

\newcommand{\bi}{\begin{itemize}}
\newcommand{\ei}{\end{itemize}}
\newcommand{\be}{\begin{enumerate}}
\newcommand{\ee}{\end{enumerate}}

\newcommand{\Afrak}{\mathfrak A}
\newcommand{\Bfrak}{\mathfrak B}
\newcommand{\Cfrak}{\mathfrak C}

\newcommand{\sbm}[1]{\left[\begin{smallmatrix}#1\end{smallmatrix}\right]}

\newcommand{\bbm}[1]{\begin{bmatrix}#1\end{bmatrix}}

\makeatletter

\def\etv{& \hskip-.3em\vrule\hskip-.3em &} 
\def\smalletv{&\vrule&} 
\def\smallcrh{\vrule height0pt depth2\ex@ width0pt
\cr\noalign{\hrule}
\vrule height6.5\ex@ depth0pt width0pt}
\newbox\smallstrutbox
\setbox\smallstrutbox=\hbox{\vrule height6pt depth1.5pt width\z@}
\def\smallstrut{\relax\ifmmode\copy\smallstrutbox\else\unhcopy\smallstrutbox\fi}

\newenvironment{sysmatrix}{
\let\|=\etv
\hskip \arraycolsep
\begin{matrix}}
{\end{matrix}
\hskip \arraycolsep
}       

\newenvironment{smallsysmatrix}{\null\,\vcenter\bgroup
\let\|=\smalletv

\def\\{\smallstrut\math@cr}
\restore@math@cr\default@tag
\baselineskip\z@skip \lineskip\z@skip \lineskiplimit\lineskip
\ialign\bgroup\hfil$\m@th\scriptstyle##$\hfil&&\thickspace\hfil
$\m@th\scriptstyle##$\hfil\crcr
\crcr\noalign{\vskip -.3\ex@}%
}{\crcr\noalign{\vskip -.2\ex@}%
\crcr\egroup\egroup\,%
}

\makeatother
